\documentclass[12pt]{article}
\usepackage[numbers]{natbib}
\usepackage{amssymb, amsmath, url, amsthm, amssymb,graphicx,multirow}
\usepackage[T1]{fontenc}
\usepackage[latin1]{inputenc}

\newtheorem{Theorem}{Theorem}[section]
\newtheorem{Lemma}{Lemma}[section]

\newtheorem{Assumption}{Assumption}[section]

\newcommand{\convP}{\stackrel{\mathcal{P}}{\longrightarrow}}

\DeclareMathOperator{\vectorization}{vec}
\renewcommand{\P}{P}
\newcommand{\inP}{\stackrel{\mathcal{P}}{\longrightarrow}}

\newcommand{\bgamma}{\boldsymbol \gamma}
\newcommand{\inD}{\stackrel{\mathcal{D}}{\longrightarrow}}
\newcommand{\as}{\stackrel{a.s.}{\longrightarrow}}
\newcommand{\ds}{ ds}
\newcommand{\im}{\boldsymbol i}
\newcommand{\Int}{\int_0^1}
\newcommand{\la}{\langle}
\newcommand{\bV}{{\bf V}}
\newcommand{\bW}{{\bf W}}
\newcommand{\ra}{\rangle}
\newcommand{\dt}{ dt}
\newcommand{\dx}{ dx}
\newcommand{\dz}{ dz}
\newcommand{\du}{ du}
\newcommand{\E}{E} 
\newcommand{\op}[1]{o_{P}\left(#1\right)}

\newcommand{\Op}[1]{O_{P}\left(#1\right)}
\newcommand{\bbeta}{\boldsymbol \beta}

\usepackage{enumerate}
\numberwithin{equation}{section}
\numberwithin{Theorem}{section}
\numberwithin{Remark}{section}
\numberwithin{Assumption}{section}
\numberwithin{Lemma}{section}
\begin{document}

\noindent {\bf\Large Detecting Changes in Functional Linear Models}
\footnote{Research partially supported by NSF grant
DMS 0905400
}
\renewcommand{\thefootnote}{\arabic{footnote}}
\\[3ex]
\\[1ex]
{\large  Lajos Horv\'ath and Ron Reeder}\\[3ex]
{\em Department of Mathematics, University of Utah, 155 South
1440
84112-0090 USA}\bigskip

\noindent{\bf Abstract} \\[1ex]
We observe two sequences of curves which are connected via an integral operator.  Our model includes linear models as well as autoregressive models in Hilbert spaces. We wish to test the null hypothesis that the operator did not change during the observation period. Our method is based on projecting the observations onto a suitably chosen finite dimensional space. The testing procedure is based on functionals of the  weighted residuals of the projections. Since the quadratic form is based on estimating the long-term covariance matrix of the residuals, we also provide some results on Bartlett-type estimators. \\


\bigskip

\section{Introduction}\label{s:intro}
Suppose $\{X_n(t),\ n=1,2,\dots,N\}$ and $\{Y_n(t),\ n=1,2,\dots,N\}$ are sequences of random functions on $[0,1]$ that satisfy the  linear relationship
\begin{equation}\label{e:model}
Y_n(t)=\int_0^1 \Psi_n(s,t) X_n(s)\ds + \epsilon_n(t).
\end{equation}
For example, $X_n(t)$ and $Y_n(t)$ may be the exchange rates of two currencies on day $n$ at time $t$, where the trading day is normalized so that $t$ ranges between $0$ and $1$. In other applications, $X_n$ can be the temperature and $Y_n$ the pollution level at a given location. If $\Psi_1=\Psi_2=\dots = \Psi_N$, we say that the model is stable. 
  However, as the underlying  conditions change, the $\Psi$'s may also change.  Our estimates for the assumed common $\Psi$ as well as our predictions and inferences based on the model would be flawed if we falsely assume that the $\Psi$'s have not changed.  To test the applicability of this model with an unchanging $\Psi$, we will test the null hypothesis,
\begin{equation}\label{eq:null}
H_0:  \Psi_1=\Psi_2=\dots = \Psi_N,
\end{equation}
against the alternative
$$
 H_A:\;\;\Psi_1=\Psi_2=\ldots = \Psi_{k^*_1}\neq\Psi_{k^*_1+1} = \ldots =\Psi_{k^*_r} \neq\Psi_{k^*_r+1}=\ldots = \Psi_N
$$
with some unknown integers $k^*_1,\ldots k^*_r$. The $k^*_i$'s are called change-points, and the alternative, $H_A$,  is that there are exactly $r$ change-points. We assume that (\ref{e:model}) and $H_0$ hold and that both $\{X_n\}$ and $\{\epsilon_n\}$ are stationary sequences.
The model with non-changing (stable) $\Psi$ has received considerable attention in the literature. If $X_n$ and $\epsilon_n$ are independent sequences of independent processes, then (\ref{e:model}) is a functional version of the classical linear model (cf.\  Cardot et al (2003), Chiou et al (2004), Cai and Hall (2006) and  Ferraty and Vieu (2006)). If $X_n=Y_{n-1}$, then we have the functional AR(1) model in (\ref{e:model}) (cf.\  Bosq (2000), Kargin and Onatsky (2008) and Horv\'ath et al (2010)). Aue et al (2011) investigated the stability of high--frequency portfolio betas in the capital asset pricing model (CAPM). CAPM is a version of the model in (\ref{e:model}) where, in our notation, a vector valued $Y_n$ is a linear combination of vector valued $X_n$'s and an additional error term.  \\

 Let $C(s,t) = \mbox{var}\left(X_n(t), X_n(s)\right)$ and $D(s,t) = \mbox{var}\left(Y_n(t), Y_n(s)\right)$.  Let $\{(v_j(s), \lambda_j),\ 1\le j \le \infty\}$ and $\{(w_i(t), \tau_i),\ 1\le i \le \infty \}$ be eigenfunction-eigenvalue pairs associated with $C(s,t)$ and $D(s,t)$ respectively.  This means that $\tau_i w_i(t)=\int_0^1 D(t,s) w_i(s)\ds$ and $\lambda_j v_j(s)=\int_0^1 C(s,t) v_j(t)\dt$.  Assume that $\lambda_j$ is the $j^{th}$ largest eigenvalue of $C(s,t)$ and that $\tau_i$ is the $i^{th}$ largest eigenvalue of $D(s,t)$.  It can be  assumed that the eigenfunctions of $C(s,t)$ are orthonormal and also that the eigenfunctions of $D(s,t)$ are  orthonormal. 
We assume that $\Psi \in L^2[0,1]^2$ and  can therefore  be expressed 
as
\begin{equation}\label{e:psi}
\Psi(s,t)=\sum_{i=1}^{\infty} \sum_{j=1}^{\infty} \psi_{i,j} v_j(s)w_i(t).
\end{equation}
Using  (\ref{e:psi}) we can write  the  model (\ref{e:model}) as
\begin{equation}\label{eq:Y}
\begin{aligned}
Y_n(t)&=\int_0^1 \Psi_n(s,t) X_n(s)\ds + \epsilon_n(t)\\
&= \int_0^1 \sum_{i=1}^{\infty} \sum_{j=1}^{\infty} \psi_{i,j} w_i(t)v_j(s) X_n(s)\ds + \epsilon_n(t)\\
&= \sum_{i=1}^{q} \sum_{j=1}^{p} \psi_{i,j} w_i(t) \int_0^1 v_j(s) X_n(s)\ds + \epsilon_n^*(t),\\
\end{aligned}
\end{equation}
where
$$
\epsilon_n^*(t)= \epsilon_n(t) +  \sum_{i=1}^{q} \sum_{j=p+1}^{\infty} \psi_{i,j} w_i(t) \int_0^1 v_j(s) X_n(s)\ds +  \sum_{i=q+1}^{\infty} \sum_{j=1}^{\infty} \psi_{i,j} w_i(t) \int_0^1 v_j(s) X_n(s)\ds.
$$
Equation (\ref{eq:Y}) means  that we keep  the parts of $Y_n$ and $X_n$ which are explained by the first $q$ and $p$ principle components.\\

To reduce the dimensionality of the model we will project both sides of (\ref{eq:Y}) onto the space spanned by the functions $\{w_i(t),\ 1\le i\le q\}$.  Doing this we obtain the  linear model
\begin{equation}\label{eq:naiveprojection}
\begin{pmatrix}
\la Y_n, w_1\ra\\
\la Y_n, w_2\ra \\
\vdots\\
\la Y_n, w_q \ra\\
\end{pmatrix}
=
\begin{pmatrix}
\psi_{1,1} & \psi_{1,2} & \dots & \psi_{1,p}\\
\psi_{2,1} & \psi_{2,2} & \dots & \psi_{2,p}\\
\vdots&\vdots &\dots &\vdots\\
\psi_{q,1} & \psi_{q,2} & \dots & \psi_{q,p}\\
\end{pmatrix}
\begin{pmatrix}
    \la X_n, v_1 \ra \\
    \la X_n, v_2 \ra\\
    \vdots\\
    \la X_n, v_p \ra\\
\end{pmatrix}
+
\begin{pmatrix}
    \la \epsilon_n^*, w_1 \ra  \\
    \la \epsilon_n^*, w_2 \ra  \\
    \vdots\\
    \la \epsilon_n^*, w_q \ra  \\
\end{pmatrix}.
\end{equation}

Instead of testing the null hypothesis, (\ref{eq:null}), exactly as it is stated, we would like to test if the coefficients $\{\psi_{i,j},\ 1\le i \le q, 1 \le j \le p\}$ remained constant during the observation period.  Essentially, we are testing the stability of $\Psi(s,t)$ over the space spanned by the most important principle components of the $X_n$'s and the $Y_n$'s.  Equation (\ref{eq:naiveprojection}) has the form of a linear model, but it is not a classical linear model because the regressors are random variables and are correlated with the errors.  Unfortunately, we cannot use (\ref{eq:naiveprojection}) directly, since the covariance functions, $D(s,t)$ and $C(s,t)$, and hence the eigenfunctions, $\{w_i(t),\ i=1,2,\dots q\}$ and $\{v_j(t),\ j=1,2,\dots,p\}$, are unknown.  Instead, we will use the estimates $\hat{D}_N(s,t)$ and $\hat{C}_N(s,t)$ and their corresponding eigenfunctions, $\{\hat{w}_{i,N}(t),\ i=1,2,\dots q\}$ and $\{\hat{v}_{j,N}(s),\ j=1,2,\dots,p\}$, where
$$
\begin{aligned}
\hat{D}_N(s,t)&=\frac{1}{N}\sum_{k=1}^N (Y_k(t)-\bar{Y}_N(t))( Y_k(s)-\bar{Y}_N(s))\quad\mbox{with}\quad
\bar{Y}_N(t)=\frac{1}{N}\sum_{i=1}^NY_i(t),\\
\hat{C}_N(s,t)&=\frac{1}{N}\sum_{k=1}^N (X_k(t)-\bar{X}_N(t))( X_k(s)-\bar{X}_N(s))\quad\mbox{with}\quad
\bar{X}_N(t)=\frac{1}{N}\sum_{i=1}^NX_i(t).
\end{aligned}
$$
Eigenfunctions corresponding to unique eigenvalues are uniquely determined up to signs.  For this reason, we cannot expect more than to have $\hat{w}_{i,N}$ be close to $\hat{d}_{i,N}w_i$ and $\hat{v}_{j,N}$ be close to $\hat{c}_{j,N}v_j$, where $\hat{d}_{i,N}$, $\hat{c}_{i,N}$ are random signs (cf. Theorem \ref{Th:eigen}). In order to obtain a linear model similar to equation (\ref{eq:naiveprojection}) that is useable, we must use our estimates for the eigenfunctions.  We replace equation (\ref{eq:Y}) with
\begin{equation}\label{eq:Y-estimated}
Y_n(t)=\sum_{i=1}^{q} \sum_{j=1}^{p} \hat{d}_{i,N} \psi_{i,j} \hat{c}_{j,N} \hat{w}_{i,N}(t) \int_0^1 \hat{v}_{j,N}(s) X_n(s)\ds + \epsilon_n^{**}(t),\\
\end{equation}
where
$$
\begin{aligned}
\hspace*{-.3 cm}\epsilon_n^{**}(t)= &\epsilon_n(t) +  \sum_{i=1}^{q} \sum_{j=p+1}^{\infty} \psi_{i,j} w_i(t) \int_0^1 v_j(s) X_n(s)\ds +  \sum_{i=q+1}^{\infty} \sum_{j=1}^{\infty} \psi_{i,j} w_i(t) \int_0^1 v_j(s) X_n(s)\ds\\
 &- \sum_{i=1}^{q} \sum_{j=1}^{p} \hat{d}_{i,N} \psi_{i,j} \hat{c}_{j,N} \hat{w}_{i,N}(t) \int_0^1 \hat{v}_{j,N}(s) X_n(s)\ds + \sum_{i=1}^{q} \sum_{j=1}^{p} \psi_{i,j} w_i(t) \int_0^1 v_j(s) X_n(s)\ds.
 \end{aligned}
$$

By projecting both sides of (\ref{eq:Y-estimated}) onto the space spanned by the functions $\{\hat{w}_{j,N}(t),\ 1\le j\le q\}$, we can replace the linear model (\ref{eq:naiveprojection}) with the empirical linear model
\begin{equation}\label{eq:projection}
\begin{split}
   \begin{pmatrix}
\la Y_n, \hat{w}_{1,N} \ra \\
\la Y_n, \hat{w}_{2,N} \ra \\
\vdots\\
 \la Y_n,  \hat{w}_{q,N}\ra \\
\end{pmatrix}
=
\begin{pmatrix}
\hat{d}_{1,N}\psi_{1,1}\hat{c}_{1,N} & \hat{d}_{1,N}\psi_{1,2}\hat{c}_{2,N} & \dots & \hat{d}_{1,N}\psi_{1,p}\hat{c}_{p,N}\\
\hat{d}_{2,N}\psi_{2,1}\hat{c}_{1,N} & \hat{d}_{2,N}\psi_{2,2}\hat{c}_{2,N} & \dots & \hat{d}_{2,N}\psi_{2,p}\hat{c}_{p,N}\\
\vdots&\vdots &\dots &\vdots\\
\hat{d}_{q,N}\psi_{q,1}\hat{c}_{1,N} & \hat{d}_{q,N}\psi_{q,2}\hat{c}_{2,N} & \dots & \hat{d}_{q,N}\psi_{q,p}\hat{c}_{p,N}\\
\end{pmatrix}
\begin{pmatrix}
    \la X_n, \hat{v}_{1,N} \ra \\
    \la X_n, \hat{v}_{2,N} \ra \\
    \vdots\\
    \la X_n, \hat{v}_{p,N} \ra \\
\end{pmatrix}
\vspace{.3 cm}\\
\hspace*{-10 cm}+
\begin{pmatrix}
    \la \epsilon_n^{**}, \hat{w}_{1,N} \ra \\
   \la  \epsilon_n^{**}, \hat{w}_{2,N}\ra  \\
    \vdots\\
  \la   \epsilon_n^{**}, \hat{w}_{q,N}\ra
\end{pmatrix}.
\end{split}
\end{equation}

The signs $\{\hat{d}_{i,N},\ 1\le i \le q\}$ and $\{\hat{c}_{j,N},\ 1\le j \le p\}$ are computed from $X_1,X_2, \ldots ,X_N$ and $Y_1,Y_2, \ldots ,Y_N$ and  they    will not change during the testing procedure.  Therefore, testing the stability of $\{\hat{d}_{i,N} \psi_{i,j} \hat{c}_{j,N},\ 1\le i \le q, 1\le j \le p\}$ is equivalent to testing the stability of $\{ \psi_{i,j} ,\ 1\le i \le q, 1\le j \le p\}$.\\

Letting $\otimes$ be the Kronecker product, we can express equation (\ref{eq:projection}) in a more condensed form:

\begin{equation}\label{eq:model-observations}
\hat{\mathbf Y}{(n)}=\hat{\mathbf Z}{(n)} {\boldsymbol \beta}+\hat{\boldsymbol \Delta}{(n)},\text{\phantom{space here}} 1\le n \le N,
\end{equation}

where

$$
\begin{aligned}
\hat{\mathbf Y}{(n)}&=
\begin{pmatrix}
\la  Y_n, \hat{w}_{1,N} \ra   \\
\la Y_n, \hat{w}_{2,N} \ra    \\
\vdots\\
\la Y_n,\ \hat{w}_{q,N}  \ra   \\
\end{pmatrix},
\quad
\hat{\boldsymbol \Delta}{(n)} =
\begin{pmatrix}
  \la   \epsilon_n^{**}, \hat{w}_{1,N} \ra  \\
 \la   \epsilon_n^{**}, \hat{w}_{2,N}   \ra \\
    \vdots\\
 \la   \epsilon_n^{**}, \hat{w}_{q,N}   \ra \\
\end{pmatrix},
\vspace*{.3 cm}\\
{\boldsymbol \beta} &= \begin{pmatrix}
\hat{d}_{1,N} \psi_{1,1} \hat{c}_{1,N}\\
\vdots\\
\hat{d}_{1,N} \psi_{1,p} \hat{c}_{p,N}\\
\hat{d}_{2,N} \psi_{2,1} \hat{c}_{1,N}\\
\vdots\\
\hat{d}_{q,N} \psi_{q,p} \hat{c}_{p,N}\\
\end{pmatrix} = \vectorization(\{\hat{d}_{i,N} \psi_{i,j} \hat{c}_{j,N},\ 1\le i \le q, 1\le j \le p\}^{T}),
\end{aligned}
$$
and
$$
\hat{\mathbf Z}{(n)}={\mathbf I}_q\otimes \hat{\mathbf M}(n)\;\;\;\;\mbox{with}\;\;\;\;
\hat{\mathbf M}(n) = \left( \la X_n, \hat{v}_{1,N} \ra,\dots, \la X_n, \hat{v}_{p,N} \ra  \right).
$$
The least squares estimator for $\bbeta$ is defined by
$$
\hat{\bbeta}_N = \left(\hat{\mathbf Z}_{N}^T\hat{\mathbf Z}_{N}\right)^{-1}\hat{\mathbf Z}_{N}^T\hat{\mathbf Y}_{N},
$$
where the  vectors $\hat{\mathbf Y}_{\lfloor Nt \rfloor}$ and the  matrices
$\hat{\mathbf Z}_{\lfloor Nt \rfloor}$ for each $t\in [0,1]$ are defined by
$$
\hat{\mathbf Y}_{\lfloor Nt \rfloor} =
\begin{pmatrix}
\hat{\mathbf Y}{(1)}\\
\hat{\mathbf Y}{(2)}\\
\vdots\\
\hat{\mathbf Y}{({\lfloor Nt \rfloor})}\\
\end{pmatrix}
\text{\quad}\text{ and }\text{\quad}
\hat{\mathbf Z}_{\lfloor Nt \rfloor} =
\begin{pmatrix}
\hat{\mathbf Z}{(1)}\\
\hat{\mathbf Z}{(2)}\\
\vdots\\
\hat{\mathbf Z}{({\lfloor Nt \rfloor})}
\end{pmatrix}.
$$
Our testing procedure is based on the cumulative sums process of the weighted residuals,
\begin{equation}\label{eq:cumsum}
\tilde{\bV}_N(t)
=N^{-1/2}\left[\sum_{n=1}^{\lfloor Nt \rfloor} \hat{\mathbf Z}^T{(n)} \tilde{\mathbf Y}{(n)} - t \sum_{n=1}^{N} \hat{\mathbf Z}^T{(n)} \tilde{\mathbf Y}{(n)}\right],
\;\;\; t\in [0,1],
\end{equation}
where $\tilde{\mathbf Y}{(n)}= \hat{\mathbf Y}{(n)}- \hat{\mathbf Z}{(n)} \hat{\boldsymbol \beta}_N, 1\leq n \leq N$ stands for the residuals.

\section{Main Results}\label{s:main}
In this section we formally state all of the assumptions that we need and then we state our main theorem.  Throughout this paper we use $\left|\cdot \right|$ to mean the absolute value of a scalar or the largest of the absolute values of the elements of a vector or matrix. It will always be clear from the context  which is meant.\\

Our first condition means that the processes $X_n$ and $\epsilon_n$  are Bernoulli shifts:
\begin{Assumption}\label{a:shifts} 
$X_n(t)$ and $\epsilon_n(t)$ can be expressed as
$$
X_n(t)=a({\boldsymbol \eta}_n(t), {\boldsymbol \eta}_{n-1}(t), \dots) \text{ \ and \ } \epsilon_n(t)=b({\boldsymbol \eta}_n(t), {\boldsymbol \eta}_{n-1}(t), \dots),
$$
for some functionals $a$ and $b$ where $\{{\boldsymbol \eta}_k, -\infty < k < \infty\}$ are iid vector-valued random functions.
\end{Assumption}

Assumption \ref{a:shifts} implies immediately that the vector-valued process $(X_n,\epsilon_n), 1\leq n <\infty$ is stationary and ergodic. If $H_0$ holds, then
$(X_n,\epsilon_n, Y_n), 1\leq n <\infty$ is also stationary and ergodic. We also require  that the processes have at least 4 moments:

\begin{Assumption}\label{a:Xmoments}
\begin{equation}\label{e:eps0}
\E X_n(t)=0 \quad \mbox{and} \quad \E\epsilon_n(t)=0,
\end{equation}
\begin{equation}\label{e:eps4}
\int_0^1 \E X_n^4(t) \dt < \infty
\quad
\mbox{and}
\quad
\int_0^1 \E \epsilon_n^4(t) \dt < \infty.
\end{equation}
\end{Assumption}

\begin{Assumption}\label{a:uncorrelated}
$X_n(t)$ and  $\epsilon_n(s)$ are uncorrelated, i.e. $\E X_n(t)\epsilon_n(s)=0$ for all $0\leq t,s \leq 1.$
\end{Assumption}

Under assumption \ref{a:shifts} one can even  have  long-range dependence among the observations. However, in this paper we are only interested in weakly dependent sequences which is stated in the next assumption:
\begin{Assumption}\label{a:k-dependent}
We assume that
\begin{equation}\label{e:dep-1}
\sum_{1\leq k<\infty}\left(\E\int_0^1\left(X_n(t)-X_{n}^{(k)}(t)\right)^4 \dt \right)^{1/4}<\infty
\end{equation}
and
\begin{equation}\label{e:dep-2}
\sum_{1\leq k<\infty}\left(\E\int_0^1\left(\epsilon_n(t)-\epsilon_{n}^{(k)}(t)\right)^4 \dt \right)^{1/4} < \infty
\end{equation}
with
$$
X_{n}^{(k)}(t)=a({\boldsymbol \eta}_n(t), {\boldsymbol \eta}_{n-1}(t), \dots, {\boldsymbol \eta}_{n-k+1}(t), {\boldsymbol \eta}_{n,n-k}^{(k)}(t), {\boldsymbol \eta}_{n,n-k-1}^{(k)}(t),\dots  )
$$
and
$$
\epsilon_{n}^{(k)}(t)=b({\boldsymbol \eta}_n(t), {\boldsymbol \eta}_{n-1}(t), \dots, {\boldsymbol \eta}_{n-k+1}(t), {\boldsymbol \eta}_{n,n-k}^{(k)}(t), {\boldsymbol \eta}_{n,n-k-1}^{(k)}(t),\dots  ),
$$
where $\{{\boldsymbol \eta}_{n,\ell}^{(k)}, -\infty < k,\ell,n < \infty\}$ are iid copies of ${\boldsymbol \eta}_0$.
\end{Assumption}

We note that, due to stationarity required by Assumption \ref{a:shifts}, it is enough to assume that (\ref{e:dep-1}) and (\ref{e:dep-2}) hold for at least one $n$. H\"ormann and Kokoszka (2010) call the processes satisfying Assumption
\ref{a:k-dependent} $L^4$-$k$-decomposable processes. This property appeared first in Ibragimov (1962) and is used several times in Billingsley (1968) in case of random variables on the line. Aue et al (2009) use an analogue of Assumption \ref{a:k-dependent} for random vectors when they derive tests to detect a change in the covariance structure of the observations. Wied at al (2011) investigate the change in the correlation under the same assumptions as in Aue at al (2009).  Aue et al (2011) provide several examples when   Assumptions  \ref{a:shifts} and
\ref{a:k-dependent} hold. For example, autoregressive, moving-average, linear
processes in Hilbert spaces satisfy this condition.  Also, the non-linear functional ARCH(1) model (cf.\ H\"ormann et al (2012)) and bilinear models (cf.\   H\"ormann and Kokoszka (2010) also satisfy Assumption \ref{a:k-dependent}. \\

Our next assumption ensures  that the $p$ and $q$ largest eigenvalues of  $C$ and $D$, respectively, are unique.
\begin{Assumption}\label{a:eigenu}
$$
\lambda_1>\lambda_2>\dots>\lambda_{p+1}
$$
and
$$
\tau_1>\tau_2>\dots>\tau_{q+1}.
$$
\end{Assumption}

\begin{Assumption}\label{a:PSImoments}
$$
\int_0^1 \int_0^1 \Psi^4(s,t) \dt \ds < \infty.
$$
\end{Assumption}

\medskip
We note that under Assumptions \ref{a:Xmoments} and \ref{a:PSImoments} we also have that $\E Y_n(t)=0$ and $\int_0^1\E Y_n^4(t)\dt <\infty.$
Let
$$
    {\boldsymbol \gamma}_{\ell} =  \vectorization \left( \{ \gamma_{\ell}(i,j), 1\le i\le q, 1\le j\le p \}^T \right),
$$
where
$$
\gamma_{\ell}(i,j)=\la X_{\ell}, v_j\ra \la \epsilon_{\ell},w_i\ra  + \la X_{\ell},v_j\ra \la X_{\ell},u_i\ra,
$$
and
$$
    u_i(s)=\sum_{r=p+1}^{\infty} \psi_{i,r}v_r(s),\ \ 1\le i \le q.
$$

Define ${\boldsymbol \Sigma}$ as
$$
    {\boldsymbol \Sigma}=\E{\boldsymbol \gamma}_0{\boldsymbol \gamma}_0^T + \sum_{l=1}^{\infty} \E{\boldsymbol \gamma}_0 {\boldsymbol \gamma}_{\ell}^T + \sum_{l=1}^{\infty} \E {\boldsymbol \gamma}_{\ell} {\boldsymbol \gamma}_0^T.
$$
We now define our detector as
$$
    V_N(t)= \tilde{{\mathbf V}}_N^T(t) \breve{\boldsymbol \Sigma}^{-1}_N \tilde{\mathbf V}_N(t),\\
$$
where $\tilde{\bV}_N(t)$ is defined in \eqref{eq:cumsum}
and $\breve{\boldsymbol \Sigma}_N$ is an estimator (up to random signs) for ${\boldsymbol \Sigma}$.  The Bartlett-type estimator that we propose for $\breve{\boldsymbol \Sigma}_N$ is a function of the estimators $\hat{v}_{j,N}(t)$ and $\hat{w}_{i,N}(t)$, which are estimators for $v(t)$ and $w(t)$ up to random signs.  For this reason, we cannot expect that $\breve{\boldsymbol \Sigma}_N$ will be close to ${\boldsymbol \Sigma}$.  The best we can expect is that ${\boldsymbol \zeta}_N \breve{{\boldsymbol \Sigma}}_N {\boldsymbol \zeta}_N$ will be close to ${\boldsymbol \Sigma}$, where ${\boldsymbol \zeta}_N$ is a matrix corresponding to the random signs, $\hat{c}_{j,N}$ and  $\hat{d}_{i,N}$.  This is described in assumption \ref{a:consistent}.

\medskip
Next we introduce the diagonal matrices $\hat{{ \bf C}}_N$ and $\hat{{\bf D}}_N$ which consists of the random signs, i.e. $\hat{{\bf C}}_N=\mbox{diag}(\hat{c}_{1,N}, \ldots,\hat{c}_{p,N} )$,  $\hat{{\bf D}}_N=\mbox{diag}(\hat{d}_{1,N}, \ldots,\hat{d}_{q,N}) $ and
${\boldsymbol \zeta}_N=\hat{{\bf D}}_N \otimes\hat{{\bf C}}_N$.
\begin{Assumption}\label{a:consistent}
    $\hat{{\boldsymbol \Sigma}}_N={\boldsymbol \zeta}_N \breve{{\boldsymbol \Sigma}}_N {\boldsymbol \zeta}_N$ is an estimator for ${\boldsymbol \Sigma}$ such that
$$
\left|\hat{{\boldsymbol \Sigma}}_N- {\boldsymbol \Sigma}\right| = \op{1}.
$$
\end{Assumption}

Note in particular that
$$
{\boldsymbol \zeta}_N {\boldsymbol \gamma}_{\ell}= \vectorization \left( \{ \hat{c}_{j,N}\hat{d}_{i,N} \gamma_{\ell}(i,j), 1\le i\le q, 1\le j\le p \}^T \right).
$$
Note also that Assumption \ref{a:consistent} and the continuous mapping theorem combined imply that $\hat{\boldsymbol \Sigma}_N^{-1} = {\boldsymbol \zeta}_N\breve{\boldsymbol \Sigma}_N^{-1}{\boldsymbol \zeta}_N \inP {\boldsymbol \Sigma}^{-1}$.\\  

Although any estimator satisfying Assumption  \ref{a:consistent} can be used, we recommend using a Bartlett-type estimator as  $\breve{\boldsymbol \Sigma}_N$, which we will describe in section \ref{s:bart}.
\begin{Theorem}\label{Th:main} If Assumptions \ref{a:shifts}-\ref{a:consistent} hold, then we have
    $$
        V_N(t) \inD \sum_{\ell=1}^{pq} {\mathcal B}_{\ell}^2(t),
    $$
    where $\{{\mathcal B}_{\ell}(t), \ell=1,\dots, pq\}$ are iid standard Brownian bridges.
\end{Theorem}

The testing procedure can be based on Theorem \ref{Th:main}, using functionals of  $V_N(t)$. The distribution of functionals of the limit was considered by Kiefer (1959) who provided formulae for the distribution functions of the supremum and $L^2$ functionals of the limit. For tables, approximations and further discussion on the distribution of functionals of the limit we refer to Aue et al (2009).

\section{Bartlett-type estimators}\label{s:bart}

In this section we discuss the estimation of the long-run covariance matrix of the sums of weakly dependent vectors. We start with estimators based on the sequence $\gamma_\ell, 1\leq \ell\leq N$. Since ${\boldsymbol\Sigma}$ is the spectral density at 0, the kernel-type  estimators discussed in Grenander and Rosenblatt (1957),  Anderson (1971), Brillinger (1975), Priestley (1981) and Rosenblatt (1985) can be used. The estimator is defined by
$$
\tilde{\boldsymbol \Sigma}_N = \sum_{k=-(N-1)}^{N-1} K(k/B_N) {\boldsymbol \phi}_{k,N},
$$
where
$$
{\boldsymbol \phi}_{k,N}= \frac{1}{N} \sum_{\ell=\max(1,1-k)}^{\min(N,N-k)} {\bgamma}_{\ell}{\bgamma}_{\ell+k}^T.
$$
The kernel $K$ satisfies the following condition:
\begin{Assumption}\label{a:K}
$$
\begin{aligned}
&(i)\;\; K(0)=1\\
&(ii)\;\; K\;\mbox{is a symmetric, Lipschitz function} \\
&(iii) \;\; K \;\mbox{has a  bounded support} \\
&(iv)\;\; \hat{K}, \;\mbox{the Fourier transform of}\; K,\;\mbox{is also Lipschitz and  integrable}
\end{aligned}
$$
\end{Assumption}
These conditions are mild, and they are satisfied by the most commonly  used kernels, like the triangle of Bartlett and the polynomial kernel of Parzen (1961, 1967).  Assumption \ref{a:K}(iii) makes the present proofs relatively technically simple and it could be replaced with the assumption that $K(x)$ decays sufficiently fast as $|x|\to \infty$. The next assumption is standard in the estimation of spectral densities and long term variances and covariances.
\begin{Assumption}\label{a:kernel}
$$
B_N\rightarrow\infty\;\;\;\mbox{and}\;\;\;B_N/N\rightarrow 0,\;\;\;\mbox{as}\;\;\;N\to \infty.
$$
\end{Assumption}
Jansson (2002) proved the consistency of covariance estimation for linear processes under the assumption $B_N=o(N^{1/2}).$ Similarly, H\"ormann and Kokoszka (2010) obtained consistency results for the estimation of the long run covariance matrices of the projections of functional observations assuming  $B_N=o(N^{1/2})$. Liu and Wu (2010) established consistency results for estimation of spectral densities under Assumption \ref{a:kernel}.

\begin{Theorem}\label{Th:bart-1} If Assumptions \ref{a:shifts}-\ref{a:k-dependent}, \ref{a:PSImoments},  \ref{a:K} and
\ref{a:kernel} hold, then
$$
\tilde{\boldsymbol \Sigma}_N\convP {\boldsymbol \Sigma}.
$$
\end{Theorem}

\medskip
We would like to point out that the proof of Theorem \ref{Th:bart-1} only requires that $\gamma_\ell$ is a Bernoulli shift with zero mean and  finite second moment for which (\ref{eq:aue}) holds.\\

The estimator, $\tilde{\boldsymbol \Sigma}_N$, cannot be computed since the variables $\bgamma_\ell$ are not observed directly and we need to estimate them from the sample. We have estimators for $v_j$ as well as for $w_i$, but we will also need an estimator for $\epsilon_\ell$. We use the residuals to get inference on $\epsilon_\ell$:
$$
\hat{\epsilon}_\ell(t)=Y_\ell(t)-\sum_{i=1}^q\sum_{j=1}^p\hat{\psi}_{i,j}\hat{w}_{i,N}(t)\la X_\ell, \hat{v}_{j,N}\ra,
$$
where $\hat{\psi}_{i,j}$ is the $(i,j)^{\mbox{{\it th}}}$ element of $\hat{\bbeta}_N$ when it is written in the matrix form, i.e.\ $\{\hat{\psi}_{i,j}, 1\leq i \leq q, 1\leq j \leq p\}=\mbox{vec}^{-1}(\hat{\bbeta}_N)$. Now $\bgamma_\ell$ will be replaced with
$$
    \hat{{\boldsymbol \gamma}}_{\ell} =  \vectorization \left( \{ \hat{\gamma}_{\ell}(i,j), 1\le i\le q, 1\le j\le p \}^T \right),
$$
where
$$
\hat{\gamma}_{\ell}(i,j)=\la X_{\ell}, \hat{v}_{j,N}\ra \la \hat{\epsilon}_{\ell},\hat{w}_{i,N}\ra.
$$
Now the Bartlett-type estimator is defined as
\begin{equation}\label{eq:hatbart}
\breve{{\boldsymbol \Sigma}}_N = \sum_{k=-(N-1)}^{N-1} K(k/B_N) \hat{{\boldsymbol \phi}}_{k,N},
\end{equation}
where
$$
\hat{{\boldsymbol \phi}}_{k,N}= \frac{1}{N} \sum_{\ell=\max(1,1-k)}^{\min(N,N-k)} \hat{{\bgamma}}_{\ell}\hat{{\bgamma}}_{\ell+k}^T.
$$
The next result states that the proposed estimator satisfies Assumption \ref{a:consistent}.
\begin{Theorem}\label{Th:bart-2} If Assumptions \ref{a:shifts}-\ref{a:PSImoments}, \ref{a:K}
hold and
\begin{equation}\label{window}
B_N\rightarrow\infty\;\;\;\mbox{and}\;\;\;B_N/N^{1/2}\rightarrow 0,
\end{equation}
then Assumption \ref{a:consistent} is satisfied.
\end{Theorem}

The estimator $\breve{{\boldsymbol \Sigma}}_N$ is based on the empirical projections $\hat{\gamma}_{\ell}(i,j)$ which will be replaced with $\hat{d}_{i,N}\hat{c}_{j,N}\gamma_{\ell}(i,j)$ in the proof. The rates in  (\ref{e:eig-w})  and (\ref{e:eig-v})  in Theorem \ref{Th:eigen} are optimal so we have assumption $B_N/N^{1/2}\to 0$ in (\ref{window}) instead of the optimal $B_N/N\to 0$.

\section{A simulation study}\label{s:empirical5}

In this section, we investigate the empirical size and power of a testing procedure using the integral of the detector, $\int|V_N(t)| \dt$, as our test statistic.  Seeking to obtain a test of size $\alpha = .01$, $.05$, or $.10$, a rejection region was chosen according to the limiting distribution of the test statistic.  Simulated data was then used to compute the outcome of the test statistic.  Iterating this procedure 10,000 times, we kept track of the proportion of times that the outcome fell in the predetermined rejection region.  When simulations are done under $H_0$, this gives us the empirical size of the test, which we expect to be close to the nominal size, $\alpha$, for large sample sizes.  When simulations are done under the alternative, $H_A$, the proportion gives us the empirical power of the test.

The $X_n(t)$'s and $\varepsilon_n(t)$'s were generated according to the distribution of independent standard Brownian bridges.  Then, using $\psi(s,t) = e^{-(s-t)^2}$, we obtained the first half of our sample according to \eqref{e:model}.  The second half of the sample was also obtained from \eqref{e:model} but used $\psi(s,t) = ce^{-(s-t)^2}$.  Thus the power of the test is a function of the parameter $c$.  In particular, when $c=1$, the null hypothesis is true. The Bartlett estimator for $\Sigma$ uses the flat-top kernel
 \begin{equation*}
 K(t)=\left\{
 \begin{array}{l}
 1  \hspace{2cm} 0\leq |t|<.1\\
   1.1-|t|\hspace{1cm} .1\leq |t|<1.1\\
   0\hspace{2 cm} |t|\geq 1.1.
 \end{array}
 \right.
 \end{equation*}
 The resulting empirical size and power are given in Tables~\ref{t:empirical1} -- \ref{t:empirical4} for various values of $p$ and $q$.  

\begin{table}[h!]
\caption{Empirical power of test (in \%) using $p=1$, $q=1$, $B_N=N^{1/3}/4$, and a flat-top kernel for $K(t)$}
\label{t:empirical1}
\centering
\begin{tabular}{|c|r|r|r|r|r|r|r|r|r|} \hline
\multirow{2}{*}{$c$} & \multicolumn{3}{|c|}{$N=100$} & \multicolumn{3}{|c|}{$N=500$} & \multicolumn{3}{|c|}{$N=1000$} \\ \cline{2-10}
& $\alpha=.01$ & $\alpha=.05$ & $\alpha=.10$& $\alpha=.01$ & $\alpha=.05$ & $\alpha=.10$& $\alpha=.01$ & $\alpha=.05$ & $\alpha=.10$\\ \hline
$1.0$	&$	0.8	$&$	5	$&$	10.2	$&$	0.9	$&$	5.1	$&$	10	$&$	1.1	$&$	5.1	$&$	10.2	$\\ \hline	
$1.2$	&$	2.5	$&$	10.1	$&$	18	$&$	15.1	$&$	34.9	$&$	46.9	$&$	35.8	$&$	60.1	$&$	71.7	 $\\ \hline	
$1.4$	&$	8.9	$&$	26.5	$&$	38.9	$&$	70.5	$&$	88.5	$&$	93.3	$&$	96.9	$&$	99.4	$&$	 99.8	$\\ \hline	
$1.6$	&$	24.1	$&$	52.2	$&$	65.5	$&$	98	$&$	99.7	$&$	99.9	$&$	100	$&$	100	$&$	100	$\\ \hline	
$1.8$	&$	46.5	$&$	75.1	$&$	85.1	$&$	100	$&$	100	$&$	100	$&$	100	$&$	100	$&$	100	$\\ \hline	
$2.0$	&$	69.7	$&$	90.7	$&$	95.3	$&$	100	$&$	100	$&$	100	$&$	100	$&$	100	$&$	100	$\\ \hline	
\end{tabular}
\end{table}

\begin{table}[h!]
\caption{Empirical power of test (in \%) using $p=1$, $q=2$, $B_N=N^{1/3}/4$, and a flat-top kernel for $K(t)$}
\label{t:empirical2}
\centering
\begin{tabular}{|c|r|r|r|r|r|r|r|r|r|} \hline
\multirow{2}{*}{$c$} & \multicolumn{3}{|c|}{$N=100$} & \multicolumn{3}{|c|}{$N=500$} & \multicolumn{3}{|c|}{$N=1000$} \\ \cline{2-10}
& $\alpha=.01$ & $\alpha=.05$ & $\alpha=.10$& $\alpha=.01$ & $\alpha=.05$ & $\alpha=.10$& $\alpha=.01$ & $\alpha=.05$ & $\alpha=.10$\\ \hline
$1.0$	&$	0.6	$&$	4.5	$&$	9.4	$&$	1	$&$	5.1	$&$	10.4	$&$	1.1	$&$	5.3	$&$	10.3	$\\ \hline	
$1.2$	&$	1.5	$&$	7.9	$&$	15.8	$&$	10.1	$&$	26.7	$&$	38.7	$&$	25.9	$&$	50.2	$&$	63	$\\ \hline	
$1.4$	&$	5.7	$&$	19.5	$&$	30.9	$&$	58	$&$	80.9	$&$	88.6	$&$	93.6	$&$	98.5	$&$	99.4	 $\\ \hline	
$1.6$	&$	15.3	$&$	40.5	$&$	55.2	$&$	95.8	$&$	99.2	$&$	99.6	$&$	100	$&$	100	$&$	100	$\\ \hline	
$1.8$	&$	35	$&$	65.4	$&$	78.2	$&$	100	$&$	100	$&$	100	$&$	100	$&$	100	$&$	100	$\\ \hline	
$2.0$	&$	56.6	$&$	83.6	$&$	91.6	$&$	100	$&$	100	$&$	100	$&$	100	$&$	100	$&$	100	$\\ \hline		
\end{tabular}
\end{table}

\begin{table}[h!]
\caption{Empirical power of test (in \%) using $p=1$, $q=3$, $B_N=N^{1/3}/4$, and a flat-top kernel for $K(t)$}
\label{t:empirical3}
\centering
\begin{tabular}{|c|r|r|r|r|r|r|r|r|r|} \hline
\multirow{2}{*}{$c$} & \multicolumn{3}{|c|}{$N=100$} & \multicolumn{3}{|c|}{$N=500$} & \multicolumn{3}{|c|}{$N=1000$} \\ \cline{2-10}
& $\alpha=.01$ & $\alpha=.05$ & $\alpha=.10$& $\alpha=.01$ & $\alpha=.05$ & $\alpha=.10$& $\alpha=.01$ & $\alpha=.05$ & $\alpha=.10$\\ \hline
$1.0$	&$	0.7	$&$	4.4	$&$	9.6	$&$	0.7	$&$	5.3	$&$	10.2	$&$	0.8	$&$	5.1	$&$	10.2	$\\ \hline	
$1.2$	&$	1.4	$&$	9.5	$&$	17.5	$&$	18.8	$&$	41.8	$&$	54.8	$&$	50	$&$	74.2	$&$	83.5	$\\ \hline	
$1.4$	&$	7.9	$&$	27.9	$&$	42.3	$&$	87.8	$&$	96.9	$&$	98.5	$&$	99.8	$&$	100	$&$	100	$\\ \hline	
$1.6$	&$	24.9	$&$	57.2	$&$	72.1	$&$	99.9	$&$	100	$&$	100	$&$	100	$&$	100	$&$	100	$\\ \hline	
$1.8$	&$	53.2	$&$	82.7	$&$	90.8	$&$	100	$&$	100	$&$	100	$&$	100	$&$	100	$&$	100	$\\ \hline	
$2.0$	&$	76	$&$	94.6	$&$	97.8	$&$	100	$&$	100	$&$	100	$&$	100	$&$	100	$&$	100	$\\ \hline	
\end{tabular}
\end{table}

\begin{table}[h!]
\caption{Empirical power of test (in \%) using $p=2$, $q=2$, $B_N=N^{1/3}/4$, and a flat-top kernel for $K(t)$}
\label{t:empirical4}
\centering
\begin{tabular}{|c|r|r|r|r|r|r|r|r|r|} \hline
\multirow{2}{*}{$c$} & \multicolumn{3}{|c|}{$N=100$} & \multicolumn{3}{|c|}{$N=500$} & \multicolumn{3}{|c|}{$N=1000$} \\ \cline{2-10}
& $\alpha=.01$ & $\alpha=.05$ & $\alpha=.10$& $\alpha=.01$ & $\alpha=.05$ & $\alpha=.10$& $\alpha=.01$ & $\alpha=.05$ & $\alpha=.10$\\ \hline
$1.0$	&$	1.4	$&$	5.9	$&$	10.7	$&$	0.9	$&$	4.8	$&$	9.6	$&$	1	$&$	4.9	$&$	10	$\\ \hline	
$1.2$	&$	2.1	$&$	8	$&$	14.1	$&$	7.6	$&$	20.8	$&$	31.2	$&$	19	$&$	39.8	$&$	52.7	$\\ \hline	
$1.4$	&$	4.9	$&$	15.5	$&$	25.4	$&$	45.8	$&$	70.2	$&$	80.5	$&$	88	$&$	96.5	$&$	98.4	 $\\ \hline	
$1.6$	&$	11.1	$&$	29.9	$&$	43.4	$&$	90.4	$&$	97.6	$&$	98.9	$&$	100	$&$	100	$&$	100	$\\ \hline	
$1.8$	&$	23.3	$&$	48.8	$&$	62.6	$&$	99.7	$&$	100	$&$	100	$&$	100	$&$	100	$&$	100	$\\ \hline	
$2.0$	&$	38.6	$&$	68	$&$	80.6	$&$	100	$&$	100	$&$	100	$&$	100	$&$	100	$&$	100	$\\ \hline		
\end{tabular}
\end{table}

\section{Random Processes in Hilbert Spaces}\label{s:random}

In this section we summarize some basic results on random variables in Hilbert spaces which are used in the proofs. Let $\|\cdot\|$ denote the $L^2$-norm of functions defined on the unit interval,  the unit square or the unit cube.

\begin{Theorem}\label{Th:sums} If Assumptions \ref{a:shifts}-\ref{a:k-dependent}
hold, then we have
\begin{equation}\label{e:sum-xe}
\left\|\frac{1}{N^{1/2}}\sum_{n=1}^NX_n(t)\epsilon_n(s) \right\|=O_\P(1),
\end{equation}
\begin{equation}\label{e:sum-xx}
\left\|\frac{1}{N^{1/2}}\sum_{n=1}^N(X_n(t)X_n(s)-C(t,s)) \right\|=O_\P(1),
\end{equation}
\begin{equation}\label{e:sum-ee}
\left\|\frac{1}{N^{1/2}}\sum_{n=1}^N(\epsilon_n(t)\epsilon_n(s)-F(t,s))) \right\|=O_\P(1),
\end{equation}
with $F(t,s)=\E(\epsilon_n(t)\epsilon_n(s))$.
If in addition Assumption \ref{a:PSImoments} is also satisfied, then
\begin{equation}\label{e:sum-yy}
\left\|\frac{1}{N^{1/2}}\sum_{n=1}^N(Y_n(t)Y_n(s)-D(t,s)) \right\|=O_\P(1).
\end{equation}
\end{Theorem}
\begin{proof} It was pointed out in H\"ormann and Kokoszka (2010) that the $k$-approximable property in Assumption \ref{a:k-dependent} implies (\ref{e:sum-xx}) and (\ref{e:sum-ee}). Using (\ref{e:model}), we get that the sums of $X_n(t)\epsilon_n(s)$ and  $ Y_n(t)Y_n(s)$ are also $k$-approximable so the rest of the result again follows from Theorem 3.1 of H\"ormann and Kokoszka (2010).
\end{proof}

\begin{Theorem}\label{Th:eigen} If Assumptions \ref{a:shifts}-\ref{a:PSImoments}
hold, then we have
\begin{equation}\label{e:eig-w}
\max_{1\leq i \leq q}\|\hat{w}_{i,N}(t) - \hat{d}_{i,N} w_i(t)\| =\Op{N^{-1/2}},
\end{equation}
\begin{equation}\label{e:eig-v}
\max_{1\leq j \leq p}\|\hat{v}_{j,N}(t) - \hat{c}_{j,N} v_j(t)\| = \Op{N^{-1/2}}
\end{equation}
and
\begin{equation}\label{e:eig-t}
\max_{1\leq i \leq q}|\hat{\tau}_{i,N}-\tau_{i}|=\Op{N^{-1/2}},
\end{equation}
\begin{equation}\label{e:eig-l}
\max_{1\leq j \leq q}|\hat{\lambda}_{j,N}-\lambda_{j}|=\Op{N^{-1/2}}.
\end{equation}
\end{Theorem}
\begin{proof} Using Corollary 1.6 of Gohberg {\it et al} (1990, p.\ 99) we get that (\ref{e:eig-w}) follows from (\ref{e:sum-yy}). According to Lemma 4.3 of Bosq (2000), (\ref{e:sum-yy}) implies (\ref{e:eig-t}). Similarly,  (\ref{e:sum-xx}) yields (\ref{e:eig-v}) and (\ref{e:eig-l}).
\end{proof}

The next result is a uniform version of Theorem \ref{Th:sums}.
\begin{Theorem}\label{Th:menshov} If Assumptions \ref{a:shifts}-\ref{a:k-dependent} and \ref{a:PSImoments}
hold, then we have
\begin{equation}\label{e:men-xe}
\max_{1\leq k \leq N}\left\|\frac{1}{N^{1/2}}\sum_{n=1}^kX_n(t)\epsilon_n(s) \right\|=O_\P(\log N),
\end{equation}
\begin{equation}\label{e:men-xx}
\max_{1\leq k \leq N}\left\|\frac{1}{N^{1/2}}\sum_{n=1}^k(X_n(t)X_n(s)-C(t,s)) \right\|=O_\P(\log N),
\end{equation}
\begin{equation}\label{e:men-ee}
\max_{1\leq k \leq N}\left\|\frac{1}{N^{1/2}}\sum_{n=1}^k(\epsilon_n(t)\epsilon_n(s)-F(t,s))) \right\|=O_\P(\log N)
\end{equation}
with $F(t,s)=\E(\epsilon_n(t)\epsilon_n(s))$.
If in addition Assumption \ref{a:PSImoments} is also satisfied, then
\begin{equation}\label{e:men-yy}
  \max_{1\leq k \leq N}\left\|\frac{1}{N^{1/2}}\sum_{n=1}^k(Y_n(t)Y_n(s)-D(t,s)) \right\|=O_\P(\log N).
\end{equation}
\end{Theorem}
\begin{proof}
Following the proof in Section A.1 in H\"ormann and Kokoszka (2010) one can easily verify that there is an integrable function $g(t,s)$ such that
$$
\E\left( \sum_{n=1}^kX_n(t)\epsilon_n(s)\right)^2\leq kg(t,s).
$$
Hence by Menshov's inequality (cf.\  M\'oricz (1976)) we have that
$$
\E\max_{1\leq k \leq N}\left( \sum_{n=1}^kX_n(t)\epsilon_n(s)\right)^2\leq (\log N)^2 Ng(t,s),
$$
implying (\ref{e:men-xe}). Similar arguments yield (\ref{e:men-xx})-(\ref{e:men-yy}).
\end{proof}

The next results establish the weak convergence of the sum of the $\gamma_\ell$'s.
\begin{Theorem}\label{aueapp} If Assumptions \ref{a:shifts}-\ref{a:k-dependent} and \ref{a:PSImoments} hold, then
$$
\frac{1}{N^{1/2}}\sum_{\ell=1}^{\lfloor Nt \rfloor}{\boldsymbol \gamma}_\ell\;\;\;\stackrel{{\cal D}^{pq}[0,1]}{\longrightarrow}\;\;\; \bW_{{\boldsymbol \Sigma}}(t),
$$
 where $\bW_{{\boldsymbol \Sigma}}$ is a $pq$ dimensional Brownian motion with zero mean and $\E(\bW_{{\boldsymbol \Sigma}} (t)\bW_{{\boldsymbol \Sigma}}(s)^T)=\min (t,s){\boldsymbol \Sigma}.$
\end{Theorem}
\begin{proof}
First we note that Assumptions \ref{a:shifts}-\ref{a:k-dependent} imply that
\begin{equation}\label{eq:aue}
\sum_{m=1}^\infty\left( \E(\gamma_\ell(i)- \gamma_\ell^{(m)}(i))^2\right)^{1/2}<\infty,
\end{equation}
where $\gamma_\ell(i)$ and $\gamma_\ell^{(m)}(i)$ are the $i^{{\mbox {\it th}}}$ coordinates of the vectors $\bgamma_\ell$ and $\bgamma_\ell^{(m)}$ with
$$
\bgamma_\ell^{(m)}=\mbox{vec}(\{\gamma_\ell^{(m)}(i,j), 1\leq i \leq q, 1\leq j \leq p\}^T),
$$
and
$$
\gamma_\ell^{(m)}(i,j)=\la X_\ell^{(m)}, v_j\ra \la \epsilon_\ell^{(m)}, w_i\ra+
\la X_\ell^{(m)}, v_j\ra \la X_\ell^{(m)}, u_i\ra .
$$
The result now follows immediately from Theorem A.1 of Aue et al (2009).
\end{proof}

\section{Proof of Theorem \ref{Th:main}}\label{s:proofs}
First we outline the proof of Theorem \ref{Th:main}. Using the definition of the residual vectors we can write that
\begin{align}\label{eq:star}
\tilde{\bV}_N(t)&=N^{-1/2}\left( \left( \hat{\mathbf Z}_{\lfloor Nt \rfloor}^T\hat{\mathbf Y}_{\lfloor Nt \rfloor} -  \hat{\mathbf Z}_{\lfloor Nt \rfloor}^T \hat{\mathbf Z}_{\lfloor Nt \rfloor}\hat{\boldsymbol \beta}_N\right) - t \left( \hat{\mathbf Z}_{N}^T\hat{\mathbf Y}_{N} -  \hat{\mathbf Z}_{N}^T \hat{\mathbf Z}_{N}\hat{\boldsymbol \beta}_N\right)\right)
\\
&=N^{-1/2}\left(\left(\hat{\mathbf Z}_{\lfloor Nt \rfloor}^T \hat{\boldsymbol \Delta}_{\lfloor Nt \rfloor} - t\hat{\mathbf Z}_{N}^T \hat{\boldsymbol \Delta}_{N}\right) + \left(\hat{\mathbf Z}_{\lfloor Nt \rfloor}^T \hat{\mathbf Z}_{\lfloor Nt \rfloor}- t\hat{\mathbf Z}_{N}^T \hat{\mathbf Z}_{N}\right)\left(\boldsymbol \beta-\hat{\boldsymbol \beta}_N\right)\right) \notag\\
&=N^{-1/2}\left(\hat{\mathbf Z}_{\lfloor Nt \rfloor}^T \hat{\boldsymbol \Delta}_{\lfloor Nt \rfloor} - t\hat{\mathbf Z}_{N}^T \hat{\boldsymbol \Delta}_{N}\right) + \left(\frac{\hat{\mathbf Z}_{\lfloor Nt \rfloor}^T \hat{\mathbf Z}_{\lfloor Nt \rfloor}- t\hat{\mathbf Z}_{N}^T \hat{\mathbf Z}_{N}}{N}\right)\left(\boldsymbol \beta-\hat{\boldsymbol \beta}_N\right)\sqrt{N},\notag
\end{align}
with
$$
\hat{\boldsymbol \Delta}_{\lfloor Nt \rfloor} =
\begin{pmatrix}
\hat{\boldsymbol \Delta}{(1)}\\
\hat{\boldsymbol \Delta}{(2)}\\
\vdots\\
\hat{\boldsymbol \Delta}{({\lfloor Nt \rfloor})}
\end{pmatrix}.
$$
We show  that
\begin{equation}\label{eq:bound}
\left(\boldsymbol \beta-\hat{\boldsymbol \beta}_N\right)\sqrt{N} = \Op{1},
\end{equation}
(cf.\  Lemma \ref{l:betahat}) and we prove in Lemma \ref{l:convergetozero} that
\begin{equation}\label{eq:littleo}
\sup_{t\in [0,1]} \left|\frac{\hat{\mathbf Z}_{\lfloor Nt \rfloor}^T \hat{\mathbf Z}_{\lfloor Nt \rfloor}- t\hat{\mathbf Z}_{N}^T \hat{\mathbf Z}_{N}}{N}\right|=\op{1}.
\end{equation}
Combining (\ref{eq:bound}) and (\ref{eq:littleo}) we conclude that
$$
\sup_{t\in [0,1]} \left|\left(\frac{\hat{\mathbf Z}_{\lfloor Nt \rfloor}^T \hat{\mathbf Z}_{\lfloor Nt \rfloor}- t\hat{\mathbf Z}_{N}^T \hat{\mathbf Z}_{N}}{N}\right)\left(\boldsymbol \beta-\hat{\boldsymbol \beta}_N\right)\sqrt{N}\right| = \op{1}.
$$
Thus we see  that $N^{-1/2}\left(\hat{\mathbf Z}_{\lfloor Nt \rfloor}^T \hat{\boldsymbol \Delta}_{\lfloor Nt \rfloor} - t\hat{\mathbf Z}_{N}^T \hat{\boldsymbol \Delta}_{N}\right)$ is the leading term while the remainder can be disregarded when considering the limiting distribution of our cumulative sum process (\ref{eq:cumsum}).\\

We  now start with the proof of (\ref{eq:littleo}).


\begin{Lemma}\label{l:converge} If Assumptions \ref{a:shifts}-\ref{a:eigenu} hold, then we have
$$
\frac{1}{k}\sum_{n=1}^k\la X_n,  v_i\ra \la X_n,  v_j\ra  \as  \lambda_i\, 1\{i=j\}\quad \mbox{as}\;\;\;k\rightarrow \infty.
$$
\end{Lemma}
\begin{proof}
We recall that $X_n(t)$ is stationary and ergodic.
Thus the ergodic theorem shows us that as $k\rightarrow \infty$
$$
\begin{aligned}
\frac{1}{k}\sum_{n=1}^k\la X_n,  v_i\ra \la X_n,\  v_j\ra  &\as \E \int_0^1 X_n(s) v_i(s)\ds \int_0^1 X_n(t) v_j(t)\dt\\
&= \E \int_0^1  v_j(t) \int_0^1  v_i(s) X_n(t)X_n(s)\ds \dt\\
&=  \int_0^1  v_j(t) \int_0^1  v_i(s) \E\left(X_n(t)X_n(s)\right)\ds \dt\\
&=  \int_0^1  v_j(t) \int_0^1  v_i(s) C(s,t)\ds \dt\\
&=  \int_0^1  v_j(t) \lambda_i  v_i(t) \dt\\
&=  \lambda_i\, 1\{i=j\},
\end{aligned}
$$
completing the proof.

\end{proof}

\begin{Lemma}\label{l:convergetozero} If Assumptions \ref{a:shifts}-\ref{a:eigenu} hold, then we have
\begin{equation}\label{e:prod-1}
\frac{1}{N}\sup_{t \in [0,1]} \left| \hat{\mathbf Z}_{\lfloor Nt \rfloor}^T \hat{\mathbf Z}_{\lfloor Nt \rfloor}- t\hat{\mathbf Z}_{N}^T \hat{\mathbf Z}_{N}\right|= \op{1}
\end{equation}
and
\begin{equation}\label{e:prod-2}
\frac{1}{N}\hat{\mathbf Z}_N^T\hat{\mathbf Z}_N \inP {\mathbf C} = {\mathbf I}_q \otimes {\boldsymbol \Lambda},
\end{equation}
 where ${\boldsymbol \Lambda}=\mbox{diag}(\lambda_1,\lambda_2, \ldots ,\lambda_p)$.
\end{Lemma}
\begin{proof} First we show that for $\delta>0$ and  $\gamma >0$  there are $K_0$ and $N_0$ such that
\begin{equation}\label{e:ineq}
\P\biggl(\sup_{K_0\leq k \leq N}\biggl|\frac{1}{k}\sum_{n=1}^k\la X_n,  \hat{v}_{i,N}\ra \la X_n,  \hat{v}_{j,N}\ra -  \lambda_i 1\{i=j\}\biggl|>\delta
 \biggl)\leq \gamma,
\end{equation}
if $N\geq N_0$. Note that by the Cauchy-Schwarz inequality we have
\begin{equation}\notag
\hspace{-2 cm}\biggl|\frac{1}{k}\sum_{n=1}^k\bigl(\la X_n,  \hat{v}_{i,N}\ra \la X_n,  \hat{v}_{j,N}\ra  - \la X_n,  \hat{c}_{i,N}{v}_i\ra \la X_n, \hat{c}_{j,N} {v}_j \ra \bigl)
\biggl|\notag
\end{equation}
\begin{equation}\notag
\leq \frac{1}{k}\sum_{n=1}^k\|X_n\|^2(\|\hat{v}_{i,N}-\hat{c}_{i,N}{v}_{i}\| +
\|\hat{v}_{j,N}-\hat{c}_{j,N}{v}_{j}\|).\notag
\end{equation}
Using  the ergodic theorem we get that
$$
\sup_{1\leq k <\infty}\frac{1}{k}\sum_{n=1}^k\|X_n\|^2<\infty \quad \mbox{a.s.},
$$
so (\ref{e:ineq}) follows from Theorem \ref{Th:eigen} and Lemma \ref{l:converge}.\\

Assume $N>N_0$.  It now follows that
$$
\begin{aligned}
&\P\left(\sup_{t\in[0,1]}\left|\hat{\mathbf Z}_{\lfloor Nt \rfloor}^T \hat{\mathbf Z}_{\lfloor Nt \rfloor}- t\hat{\mathbf Z}_{N}^T \hat{\mathbf Z}_{N}\right| >4 N\delta \right)\\
&\le\P\left(\sup_{0\le t\le {K_0}/{N}}\left|\hat{\mathbf Z}_{\lfloor Nt \rfloor}^T \hat{\mathbf Z}_{\lfloor Nt \rfloor}- t\hat{\mathbf Z}_{N}^T \hat{\mathbf Z}_{N}\right| + \sup_{ {K_0}/{N}\leq t \leq 1}\left|\hat{\mathbf Z}_{\lfloor Nt \rfloor}^T \hat{\mathbf Z}_{\lfloor Nt \rfloor}- t\hat{\mathbf Z}_{N}^T \hat{\mathbf Z}_{N}\right|>4 N\delta \right)\\
&\le\P\left(\sup_{0\leq t\le {K_0}/{N}}\left|\frac{\hat{\mathbf Z}_{\lfloor Nt \rfloor}^T \hat{\mathbf Z}_{\lfloor Nt \rfloor}- t\hat{\mathbf Z}_{N}^T \hat{\mathbf Z}_{N}}{N}\right| +
\sup_{ {K_0}/{N}\le t \le 1}\left|\frac{\hat{\mathbf Z}_{\lfloor Nt \rfloor}^T \hat{\mathbf Z}_{\lfloor Nt \rfloor}}{Nt}- \frac{\hat{\mathbf Z}_{N}^T \hat{\mathbf Z}_{N}}{N}\right|>4 \delta \right)\\
&\le\P\left(\sup_{0\le  t\le {K_0}/{N}}\left|\frac{\hat{\mathbf Z}_{\lfloor Nt \rfloor}^T \hat{\mathbf Z}_{\lfloor Nt \rfloor}- t\hat{\mathbf Z}_{N}^T \hat{\mathbf Z}_{N}}{N}\right| +
\sup_{ {K_0}/{N}\leq t \leq 1}\left|\frac{\hat{\mathbf Z}_{\lfloor Nt \rfloor}^T \hat{\mathbf Z}_{\lfloor Nt \rfloor}}{Nt}- {\mathbf C}\right| + \left|\frac{\hat{\mathbf Z}_{N}^T \hat{\mathbf Z}_N}{N}-{\mathbf C}\right|>4 \delta \right)\\
&\le\P\left(\max_{1\leq k \leq K_0}\left|\frac{\hat{\mathbf Z}_{k}^T \hat{\mathbf Z}_{k}}{N}\right| + \left| \frac{K_0\hat{\mathbf Z}_{N}^T \hat{\mathbf Z}_{N}}{N^2}\right| +
\max_{ {K_0}\le k \le N}\left|\frac{\hat{\mathbf Z}_{k}^T \hat{\mathbf Z}_{k}}{k}- {\mathbf C}\right| + \left|\frac{\hat{\mathbf Z}_{N}^T \hat{\mathbf Z}_N}{N}-{\mathbf C}\right|>4 \delta \right).
\end{aligned}
$$
For every $K_0$ we have that $\P(\max_{1\leq k \leq K_0}|{\hat{\mathbf Z}_{k}^T \hat{\mathbf Z}_{k}}|/{N}>\delta)\rightarrow 0$ and by  (\ref{e:ineq})   $\P(|{K_0\hat{\mathbf Z}_{N}^T \hat{\mathbf Z}_{N}}|/{N^2}> \delta)\rightarrow 0$ as $N\rightarrow \infty$. Using   (\ref{e:ineq}) again  we conclude $\P(\max_{ {K_0}\le k \le N}|{\hat{\mathbf Z}_{k}^T \hat{\mathbf Z}_{k}}/{k}- {\mathbf C}|>\delta)\leq \gamma.$ Since $\gamma$ and $\delta$ can be chosen as small as we wish, Lemma \ref{l:convergetozero} is established.

\end{proof}
We continue with the properties of $\hat{\mathbf Z}_{{\lfloor Nt \rfloor}}^T \hat{\boldsymbol \Delta}_{{\lfloor Nt \rfloor}}$.  First we observe that
\begin{equation}\label{eq:zzz}
\begin{aligned}
\hat{\mathbf Z}_{{\lfloor Nt \rfloor}}^T \hat{\boldsymbol \Delta}_{{\lfloor Nt \rfloor}} &= \sum_{\ell=1}^{\lfloor Nt \rfloor} \hat{\mathbf Z}^T{(\ell)} \hat{\boldsymbol \Delta}{(\ell)}\\
&=\sum_{\ell=1}^{\lfloor Nt \rfloor} \vectorization \left(\{\la X_{\ell}, \hat{v}_{j,N}\ra  \la \epsilon^{**}_{\ell},\ \hat{w}_{i,N}\ra, 1\le i \le q, 1\le j \le p\}^T\right).\\
\end{aligned}
\end{equation}
We note that
$$
\epsilon_{\ell}^{**}(t)=\epsilon_{\ell}(t) + \eta_{{\ell},1}(t) + \eta_{{\ell},2}(t) + \eta_{{\ell},3}(t) + \eta_{{\ell},4}(t) + \eta_{{\ell},5}(t),
$$
with
$$
\begin{aligned}
\eta_{n,1}(t)&=\sum_{i=q+1}^{\infty} \sum_{j=1}^{\infty} \psi_{i,j} w_i(t) \int_0^1 v_j(s) X_n(s)\ds
=\sum_{i=q+1}^{\infty} \sum_{j=1}^{\infty} \psi_{i,j} w_i(t) \la v_j, X_n\ra,\\
\eta_{n,2}(t)&=\sum_{i=1}^{q} \sum_{j=p+1}^{\infty} \psi_{i,j} w_i(t) \int_0^1 v_j(s) X_n(s)\ds
=\sum_{i=1}^{q} \sum_{j=p+1}^{\infty} \psi_{i,j} w_i(t) \la v_j, X_n\ra,\\
\eta_{n,3}(t)&=\sum_{i=1}^{q} \sum_{j=1}^{p} \hat{d}_{i,N} \psi_{i,j} \hat{c}_{j,N} \hat{d}_{i,N} w_i(t) \int_0^1 \left( \hat{c}_{j,N} v_j(s)-\hat{v}_{j,N}(s)\right) X_n(s)\ds\\
&=\sum_{i=1}^{q} \sum_{j=1}^{p} \hat{d}_{i,N} \psi_{i,j} \hat{c}_{j,N} \hat{d}_{i,N} w_i(t) \la \left( \hat{c}_{j,N} v_j-\hat{v}_{j,N}\right),\ X_n\ra,\\
\eta_{n,4}(t)&=\sum_{i=1}^{q} \sum_{j=1}^{p} \hat{d}_{i,N} \psi_{i,j} \hat{c}_{j,N} \left(\hat{d}_{i,N} w_i(t)-\hat{w}_{i,N}(t)\right) \int_0^1 \hat{c}_{j,N} v_j(s) X_n(s)\ds\\
&=\sum_{i=1}^{q} \sum_{j=1}^{p} \hat{d}_{i,N} \psi_{i,j} \hat{c}_{j,N} \left(\hat{d}_{i,N} w_i(t)-\hat{w}_{i,N}(t)\right) \la \hat{c}_{j,N} v_j,  X_n\ra,\\
\eta_{n,5}(t)&= \sum_{i=1}^{q} \sum_{j=1}^{p} \hat{d}_{i,N} \psi_{i,j} \hat{c}_{j,N} \left(\hat{w}_{i,N}(t)- \hat{d}_{i,N} w_i(t)\right) \int_0^1\left( \hat{c}_{j,N} v_j(s)-\hat{v}_{j,N}(s)\right) X_n(s)\ds\\
&= \sum_{i=1}^{q} \sum_{j=1}^{p} \hat{d}_{i,N} \psi_{i,j} \hat{c}_{j,N} \left(\hat{w}_{i,N}(t)- \hat{d}_{i,N} w_i(t)\right) \la \left( \hat{c}_{j,N} v_j-\hat{v}_{j,N}\right),  X_n\ra .\\
\end{aligned}
$$
In particular, we can write
\begin{equation}\label{eq:6partsprod}
\begin{aligned}
\la \epsilon^{**}_{\ell},\ \hat{w}_{i,N}\ra &= \la\epsilon_{\ell}, \hat{w}_{i,N}\ra+ \la\eta_{{\ell},1}, \hat{w}_{i,N}\ra+\la\eta_{{\ell},2}, \hat{w}_{i,N}\ra\\
&\qquad +\la\eta_{{\ell},3}, \hat{w}_{i,N}\ra+ \la\eta_{{\ell},4}, \hat{w}_{i,N}\ra+ \la\eta_{{\ell},5}, \hat{w}_{i,N}\ra.
\end{aligned}
\end{equation}
We show that $\hat{\mathbf Z}_{{\lfloor Nt \rfloor}}^T \hat{\boldsymbol \Delta}_{{\lfloor Nt \rfloor}}$ can be written as the sum of weakly dependent variables and an additional term which is just $t$ times a random variable matrix. The additional term reflects the replacement of $\Psi$ with a finite sum and the estimation of the eigenfunctions $\{w_i, 1\leq i \leq q\}$ and $\{v_j, 1\leq i \leq p\}$. The drift term is given by
$$
{\mathbf R}_N = \vectorization \left(\{R_{N}(i,j), 1\le i\le q, 1\le j \le p\}^T\right),
$$
where
$$
\begin{aligned}
R_N(i,j)&=R_N^{(1)}(i,j)+R_N^{(2)}(i,j)+R_N^{(3)}(i,j)+R_N^{(4)}(i,j),\\
R_N^{(1)}(i,j)&= \hat{c}_{j,N} \lambda_j \sum_{r=q+1}^{\infty} \psi_{r,j} \Int w_r(x)\left(\hat{w}_{i,N}(x)-\hat{d}_{i,N} w_i(x)\right)\dx,\\
R_N^{(2)}(i,j)&=\hat{d}_{i,N} \Int \left(\hat{v}_{j,N}(z) - \hat{c}_{j,N}v_j(z)\right) \sum_{n=p+1}^{\infty} \psi_{i,n}  \lambda_n v_n(z)\dz,\\
R_N^{(3)}(i,j)&= \hat{c}_{j,N} \hat{d}_{i,N} \lambda_j \sum_{n=1}^p \psi_{i,n} \hat{c}_{n,N} \Int \left(\hat{c}_{n,N}v_{n}(s) - \hat{v}_{n,N}(s)\right)v_j(s)\ds,\\
R_N^{(4)}(i,j)&=\hat{c}_{j,N} \hat{d}_{i,N} \lambda_j \sum_{r=1}^q \hat{d}_{r,N} \psi_{r,j} \Int w_i(x) \left(\hat{d}_{r,N} w_r(x) - \hat{w}_{r,N}(x)\right)\dx.
\end{aligned}
$$



\begin{Lemma}\label{l:epsilon} If Assumptions \ref{a:shifts}-\ref{a:eigenu} hold, then we have
$$
\sup_{t\in[0,1]} \left|\sum_{\ell = 1}^{\lfloor Nt \rfloor} \la X_{\ell}, \hat{v}_{j,N}\ra \la\epsilon_{\ell},\ \hat{w}_{i,N}\ra  - \hat{c}_{j,N}\hat{d}_{i,N}T_{\lfloor Nt \rfloor}^{(1)}(i,j)\right|=\Op{\log N},
$$
where
$$
T_{\lfloor Nt \rfloor}^{(1)}(i,j)=\sum_{\ell = 1}^{\lfloor Nt \rfloor} \la X_{\ell}, v_j\ra \la\epsilon_{\ell},\ w_i\ra .
$$
\end{Lemma}
\begin{proof} We note that
$$
\begin{aligned}
&\sup_{t\in[0,1]} \left|\sum_{\ell = 1}^{\lfloor Nt \rfloor} \la X_{\ell}, \hat{v}_{j,N}\ra  \la \epsilon_{\ell}, \hat{w}_{i,N}\ra  - \sum_{\ell = 1}^{\lfloor Nt \rfloor} \la X_{\ell}, \hat{c}_{j,N} v_j\ra  \la \epsilon_{\ell}, \hat{d}_{i,N} w_i\ra\right|\\
&\leq\sup_{t\in[0,1]} \left|\sum_{\ell = 1}^{\lfloor Nt \rfloor} \la X_{\ell}, \hat{v}_{j,N}- \hat{c}_{j,N} v_j\ra \la \epsilon_{\ell}, \hat{w}_{i,N}\ra \right| + \sup_{t\in[0,1]}\left|\sum_{\ell = 1}^{\lfloor Nt \rfloor} \la X_{\ell}, \hat{c}_{j,N} v_j\ra  \la \epsilon_{\ell}, \hat{w}_{i,N} - \hat{d}_{i,N} w_i\ra \right|.
\end{aligned}
$$
Using the Cauchy-Schwarz inequality we get that
$$
\begin{aligned}
\sup_{t\in[0,1]} &\left|\sum_{\ell = 1}^{\lfloor Nt \rfloor} \la X_{\ell}, \hat{v}_{j,N}- \hat{c}_{j,N} v_j\ra \la \epsilon_{\ell}, \hat{w}_{i,N}\ra \right|\\
&\leq \sup_{t\in[0,1]}\left\|\sum_{\ell = 1}^{\lfloor Nt \rfloor} X_{\ell}(x)\epsilon_\ell(s)\right\|\left( \|\hat{v}_{j,N}- \hat{c}_{j,N} v_j \|\|\hat{w}_{i,N}\|\right)\\
&=\Op{\log N},
\end{aligned}
$$
on account of (\ref{e:eig-v}), Theorem  \ref{Th:menshov} and $\|\hat{w}_{i,N}\|=1$. Similar arguments give that
$$
\sup_{t\in[0,1]}\left|\sum_{\ell = 1}^{\lfloor Nt \rfloor} \la X_{\ell}, \hat{c}_{j,N} v_j\ra  \la \epsilon_{\ell}, \hat{w}_{i,N} - \hat{d}_{i,N} w_i\ra \right|=\Op{\log N},
$$
completing the proof of the lemma.
\end{proof}

\begin{Lemma}\label{l:eta_1} If Assumptions \ref{a:shifts}-\ref{a:PSImoments} hold, then we have
$$
\sup_{t\in[0,1]} \left| \sum_{\ell=1}^{\lfloor Nt \rfloor} \la X_{\ell}, \hat{v}_{j,N}\ra \la \eta_{\ell, 1}, \hat{w}_{i,N}\ra  - {\lfloor Nt \rfloor} R_N^{(1)}(i,j)\right| =\Op{\log N}.
$$
\end{Lemma}
\begin{proof}
Using the orthogonality of the $w_i$'s we get that
$$
\begin{aligned}
\la \eta_{\ell, 1}, w_i\ra&=\Int \sum_{r=q+1}^{\infty} \sum_{n=1}^{\infty} \psi_{r,n} w_r(x) \left\{\Int v_n(s) X_{\ell}(s)\ds\right\} w_i(x)\dx \\
&= \sum_{r=q+1}^{\infty} \sum_{n=1}^{\infty} \psi_{r,n} \Int v_n(s) X_{\ell}(s)\ds \left(\Int w_i(x)w_r(x)\dx\right) \\
&=0.
\end{aligned}
$$
Therefore we have
$$
\begin{aligned}
\la X_{\ell},& \hat{v}_{j,N}\ra \la \eta_{\ell, 1}, \hat{w}_{i,N}\ra \\
&=\la X_{\ell}, \hat{v}_{j,N}- \hat{c}_{j,N}v_j\ra \la\eta_{\ell, 1}, \hat{w}_{i,N}-\hat{d}_{i,N}w_i\ra
+ \la X_{\ell},  \hat{c}_{j,N}v_j\ra \la\eta_{\ell, 1}, \hat{w}_{i,N}-\hat{d}_{i,N}w_i\ra .
\end{aligned}
$$
Now,
$$
\begin{aligned}
\sum_{\ell=1}^{\lfloor Nt \rfloor} \la X_{\ell}, \hat{c}_{j,N}v_j\ra \la\eta_{\ell, 1}, \hat{w}_{i,N}-\hat{d}_{i,N}w_i\ra
=A_{\lfloor Nt \rfloor}^{(1)}+A_{\lfloor Nt \rfloor}^{(2)},
\end{aligned}
$$
where
$$
\begin{aligned}
A_{\lfloor Nt \rfloor}^{(1)}&=\hat{c}_{j,N}\Int \Int \Int    v_j(z) \sum_{r=q+1}^{\infty} \sum_{n=1}^{\infty} \psi_{r,n} w_r(x)  v_n(s) \left(\hat{w}_{i,N}(x)-\hat{d}_{i,N}w_i(x)\right) \\
&\qquad \times \left(\sum_{\ell=1}^{\lfloor Nt \rfloor} X_{\ell}(z) X_{\ell}(s) - {\lfloor Nt \rfloor}C(z,s)\right)\dz \ds \dx\\
\end{aligned}
$$
and
$$
\begin{aligned}
A_{\lfloor Nt \rfloor}^{(2)}&={\lfloor Nt \rfloor}\hat{c}_{j,N}\int_0^1 \int_0^1 \int_0^1    v_j(z) \sum_{r=q+1}^{\infty} \sum_{n=1}^{\infty} \psi_{r,n} w_r(x)  v_n(s) \left(\hat{w}_{i,N}(x)-\hat{d}_{i,N}w_i(x)\right)C(z,s)\dz \ds \dx\\
&={\lfloor Nt \rfloor}\hat{c}_{j,N}\Int  \lambda_j \sum_{r=q+1}^{\infty}  \psi_{r,j} w_r(x)  \left(\hat{w}_{i,N}(x)-\hat{d}_{i,N}w_i(x)\right) \dx\\
&={\lfloor Nt \rfloor} R_N^{(1)}(i,j),
\end{aligned}
$$
where we used that the $v_j$'s  are orthonormal eigenfunctions of $C$.\\

Applying again (\ref{e:eig-w}) and (\ref{e:men-xx}) we conclude
\begin{equation}\notag
\begin{aligned}
\sup_{t\in [0,1]} \left| A_{\lfloor Nt \rfloor}^{(1)} \right|
=\Op{\log N}.
\end{aligned}
\end{equation}

Finally, using Theorems \ref{Th:eigen} and \ref{Th:menshov}, we obtain that
$$
\begin{aligned}
&\sup_{t\in [0,1]}\left|\sum_{\ell=1}^{\lfloor Nt \rfloor} \la X_{\ell}, \hat{v}_{j,N}- \hat{c}_{j,N}v_j\ra \la\eta_{\ell, 1}, \hat{w}_{i,N}-\hat{d}_{i,N}w_i\ra\right|\\
&\le\left|\left|\hat{v}_{j,N}(z)- \hat{c}_{j,N}v_j(z)\right|\right| \left|\left|\sum_{r=q+1}^{\infty} \sum_{n=1}^{\infty} \psi_{r,n} w_r(x)  v_n(s) \right|\right|\left|\left|\hat{w}_{i,N}(x)-\hat{d}_{i,N}w_i(x)\right|\right|  \\
&\qquad \times
 \sup_{t\in [0,1]}\left|\left|\sum_{\ell=1}^{\lfloor Nt \rfloor} X_{\ell}(z) X_{\ell}(s) \right|\right|\\
&=\Op{N^{-1/2}}  O(1) \Op{N^{-1/2}}  \Op{N}.\\
\end{aligned}
$$

\end{proof}
\begin{Lemma}\label{l:eta_2} If Assumptions \ref{a:shifts}-\ref{a:PSImoments} hold, then we have
$$
\sup_{t\in[0,1]} \left| \sum_{\ell=1}^{\lfloor Nt \rfloor} \la X_{\ell}, \hat{v}_{j,N}\ra \la \eta_{\ell, 2}, \hat{w}_{i,N}\ra - \left( \hat{c}_{j,N} \hat{d}_{i,N} T_{\lfloor Nt \rfloor}^{(2)}(i,j) + {\lfloor Nt \rfloor}R_N^{(2)}(i,j)\right) \right|=\Op{\log N},
$$
where
$$
T_{\lfloor Nt \rfloor}^{(2)}(i,j)=\sum_{\ell=1}^{\lfloor Nt \rfloor} \int_0^1 \int_0^1 \left(X_{\ell}(s) X_{\ell}(z) - C(z,s)\right)\sum_{r=p+1}^{\infty} \psi_{ir} v_r(s) v_j(z) \dz \ds.
$$
\end{Lemma}
\begin{proof} First we write
$$
\sum_{\ell=1}^{\lfloor Nt \rfloor} \la X_{\ell}, \hat{v}_{j,N} \ra \la\eta_{\ell, 2}, \hat{w}_{i,N}\ra = A_{\lfloor Nt \rfloor}^{(3)}+A_{\lfloor Nt \rfloor}^{(4)}+A_{\lfloor Nt \rfloor}^{(5)}+A_{\lfloor Nt \rfloor}^{(6)},
$$
where
$$
\begin{aligned}
A_{\lfloor Nt \rfloor}^{(3)}&=\hat{c}_{j,N}\hat{d}_{i,N} \sum_{\ell=1}^{\lfloor Nt \rfloor} \la X_{\ell}, v_j\ra \la \eta_{\ell, 2}, w_i\ra ,\\
A_{\lfloor Nt \rfloor}^{(4)}&=\hat{c}_{j,N}\sum_{\ell=1}^{\lfloor Nt \rfloor} \la X_{\ell}, v_j\ra \la \eta_{\ell, 2}, \hat{w}_{i,N}-\hat{d}_{i,N} w_i\ra ,\\
A_{\lfloor Nt \rfloor}^{(5)}&=\hat{d}_{i,N}\sum_{\ell=1}^{\lfloor Nt \rfloor} \la X_{\ell}, \hat{v}_{j,N}-\hat{c}_{j,N} v_j\ra \la \eta_{\ell, 2}, w_i\ra,  \\
A_{\lfloor Nt \rfloor}^{(6)}&=\sum_{\ell=1}^{\lfloor Nt \rfloor} \la X_{\ell}, \hat{v}_{j,N}-\hat{c}_{j,N} v_j\ra \la \eta_{\ell, 2}, \hat{w}_{i,N}-\hat{d}_{i,N} w_i\ra . \\
\end{aligned}
$$
The orthonormality of $\{w_i,\ 1\le i < \infty\}$ shows that  for all $1\leq i \leq q$
$$
\begin{aligned}
\la\eta_{\ell, 2}, w_i\ra &=\int_0^1 \sum_{r=1}^q \sum_{n=p+1}^{\infty} \psi_{r,n} w_r(x)\left\{ \Int v_n(s) X_{\ell}(s) \ds\right\} w_i(x) \dx\\
&=\sum_{n=p+1}^{\infty} \psi_{i,n} \Int v_n(s) X_{\ell}(s) \ds.
\end{aligned}
$$
Therefore, using again that the $v_j$'s are orthonormal eigenfunctions of $C$ we have
$$
\begin{aligned}
A_{\lfloor Nt \rfloor}^{(3)}&=
=\hat{c}_{j,N}\hat{d}_{i,N} \sum_{n=p+1}^{\infty} \psi_{i,n} \int_0^1 \int_0^1  v_j(z)  v_n(s) \left(\sum_{\ell=1}^{\lfloor Nt \rfloor} X_{\ell}(z) X_{\ell}(s) -{\lfloor Nt \rfloor}C(s,z)\right)\ds \dz\\
&=\hat{c}_{j,N}\hat{d}_{i,N}T_{\lfloor Nt \rfloor}^{(2)}(i,j).
\end{aligned}
$$
We decompose $A_{\lfloor Nt \rfloor}^{(4)}$ as
$$
\begin{aligned}
A_{\lfloor Nt \rfloor}^{(4)}&=\hat{c}_{j,N}\sum_{\ell=1}^{\lfloor Nt \rfloor} \int_0^1 \int_0^1 X_{\ell}(z)v_j(z)\left(\hat{w}_{i,N}(x)-\hat{d}_{i,N}w_i(x)\right)\sum_{r=1}^q \sum_{n=p+1}^\infty \psi_{r,n} w_r(x)  \\
&\qquad \times \int_0^1 v_n(s)X_{\ell}(s)\ds \dz\dx\\
&=\hat{c}_{j,N} \int_0^1 \int_0^1 \int_0^1 v_j(z)\left(\hat{w}_{i,N}(x)-\hat{d}_{i,N}w_i(x)\right)\sum_{r=1}^q \sum_{n=p+1}^\infty \psi_{r,n} w_r(x)  v_n(s) \\
&\qquad \times \sum_{\ell=1}^{\lfloor Nt \rfloor} X_{\ell}(s)X_{\ell}(z) \ds \dz\dx\\
&=A_{{\lfloor Nt \rfloor},1}^{(4)}+A_{{\lfloor Nt \rfloor},2}^{(4)},
\end{aligned}
$$
where
$$
\begin{aligned}
A_{{\lfloor Nt \rfloor},1}^{(4)}&=\hat{c}_{j,N} \int_0^1 \int_0^1 \int_0^1 v_j(z)\left(\hat{w}_{i,N}(x)-\hat{d}_{i,N}w_i(x)\right)\sum_{r=1}^q \sum_{n=p+1}^\infty \psi_{r,n} w_r(x)  v_n(s) \\
&\qquad \times \left(\sum_{\ell=1}^{\lfloor Nt \rfloor} X_{\ell}(s)X_{\ell}(z) - {\lfloor Nt \rfloor}C(s,z)\right) \ds \dz\dx\\
\end{aligned}
$$
and
$$
\begin{aligned}
A_{{\lfloor Nt \rfloor},2}^{(4)}&=\hat{c}_{j,N} {\lfloor Nt \rfloor} \int_0^1 \int_0^1 \int_0^1 v_j(z)\left(\hat{w}_{i,N}(x)-\hat{d}_{i,N}w_i(x)\right)  \\
&\qquad \times \sum_{r=1}^q \sum_{n=p+1}^\infty \psi_{r,n} w_r(x)  v_n(s)C(s,z) \ds \dz\dx\\
&=\hat{c}_{j,N} {\lfloor Nt \rfloor} \Int \Int  \lambda_j v_j(s)\left(\hat{w}_{i,N}(x)-\hat{d}_{i,N}w_i(x)\right)\sum_{r=1}^q \sum_{n=p+1}^\infty \psi_{r,n} w_r(x)  v_n(s)\ds\dx \\
&=0,
\end{aligned}
$$
using again that the $v_j$'s are eigenfunctions of $C$. Therefore we obtain
$$
\begin{aligned}
\sup_{t\in [0,1]} \left|A_{\lfloor Nt \rfloor}^{(4)}\right|&=\sup_{t\in [0,1]} \left|A_{{\lfloor Nt \rfloor},1}^{(4)}\right|\\
&\le \left\|\hat{w}_{i,N}(x)-\hat{d}_{i,N}w_i(x)\right\| \left\| v_j(z)\sum_{r=1}^q \sum_{n=p+1}^\infty \psi_{r,n} w_r(x)  v_n(s)\right\|  \\
&\qquad \times \sup_{t\in [0,1]} \left\|\sum_{\ell=1}^{\lfloor Nt \rfloor} X_{\ell}(s)X_{\ell}(z) - {\lfloor Nt \rfloor}C(s,z) \right\|\\
&=\Op{N^{-1/2}} \Op{1} \Op{N^{1/2}\log N}.
\end{aligned}
$$
Similar arguments give
$$
\begin{aligned}
A_{\lfloor Nt \rfloor}^{(5)}&=\hat{d}_{i,N} \sum_{\ell=1}^{\lfloor Nt \rfloor} \int_0^1 \int_0^1 X_{\ell}(z)\left(\hat{v}_{j,N}(z) - \hat{c}_{j,N}v_j(z)\right) w_i(x) \sum_{r=1}^q \sum_{n=p+1}^{\infty} \psi_{r,n} w_r(x) \ \\
&\qquad \times \int_0^1 v_n(s)X_{\ell}(s)\ds \dz \dx\\
&=A_{{\lfloor Nt \rfloor},1}^{(5)}+A_{{\lfloor Nt \rfloor},2}^{(5)},\\
\end{aligned}
$$
where
$$
\begin{aligned}
A_{{\lfloor Nt \rfloor},1}^{(5)}
&=\hat{d}_{i,N} \int_0^1 \int_0^1  \left(\hat{v}_{j,N}(z) - \hat{c}_{j,N}v_j(z)\right)  \sum_{n=p+1}^{\infty} \psi_{i,n}   v_n(s)\\
&\qquad \times \left(\sum_{\ell=1}^{\lfloor Nt \rfloor}  X_{\ell}(s)X_{\ell}(z) - {\lfloor Nt \rfloor}C(s,z)\right)\ds \dz\\
\end{aligned}
$$
and
$$
\begin{aligned}
A_{{\lfloor Nt \rfloor},2}^{(5)}&=\hat{d}_{i,N} {\lfloor Nt \rfloor} \int_0^1 \int_0^1 \int_0^1 \left(\hat{v}_{j,N}(z) - \hat{c}_{j,N}v_j(z)\right) w_i(x) \sum_{r=1}^q \sum_{n=p+1}^{\infty} \psi_{r,n} w_r(x)  v_n(s)C(s,z)\ds \dz \dx\\
&={\lfloor Nt \rfloor} R_N^{(2)}(i,j).
\end{aligned}
$$
Repeating our previous arguments we get that
$$
\begin{aligned}
\sup_{t\in [0,1]} \left|A_{{\lfloor Nt \rfloor},1}^{(5)}\right| &\le \left\|\hat{v}_{j,N}(z) - \hat{c}_{j,N}v_j(z)\right\| \left\| \sum_{n=p+1}^{\infty} \psi_{i,n}   v_n(s)\right\|
 \sup_{t\in [0,1]} \left\|\sum_{\ell=1}^{\lfloor Nt \rfloor}  X_{\ell}(s)X_{\ell}(z) - {\lfloor Nt \rfloor}C(s,z)\right\|\\
&=\Op{N^{-1/2}}  O(1)  \Op{N^{1/2}\log N}.\\
\end{aligned}
$$
Similarly, using the Cauchy-Schwarz inequality with  (\ref{e:sum-xx}) and Theorem \ref{Th:eigen},  we conclude that
$$
\begin{aligned}
\sup_{t\in [0,1]}&|A_{\lfloor Nt \rfloor}^{(6)}|
O_P(1),
\end{aligned}
$$
completing the proof of the lemma.
\end{proof}
\begin{Lemma}\label{l:eta_3} If Assumptions \ref{a:shifts}-\ref{a:PSImoments} hold, then we have
$$
\sup_{t\in[0,1]}  \left| \sum_{\ell=1}^{\lfloor Nt \rfloor}  \la X_{\ell}, \hat{v}_{j,N}\ra \la\eta_{\ell,3}, \hat{w}_{i,N}\ra - {\lfloor Nt \rfloor} R_N^{(3)}(i,j) \right|=\Op{\log N}.
$$
\end{Lemma}
\begin{proof} We write
$$
\sum_{\ell=1}^{\lfloor Nt \rfloor}  \la X_{\ell}, \hat{v}_{j,N}\ra \la \eta_{\ell,3},\ \hat{w}_{i,N}\ra  = A_{\lfloor Nt \rfloor}^{(7)}+ A_{\lfloor Nt \rfloor}^{(8)}+A_{\lfloor Nt \rfloor}^{(9)},
$$
where
$$
\begin{aligned}
A_{\lfloor Nt \rfloor}^{(7)}&= \sum_{\ell=1}^{\lfloor Nt \rfloor}  \la X_{\ell}, \hat{v}_{j,N} - \hat{c}_{j,N} v_j\ra \la \eta_{\ell,3}, \hat{w}_{i,N}\ra,\\
A_{\lfloor Nt \rfloor}^{(8)}&= \hat{c}_{j,N}\sum_{\ell=1}^{\lfloor Nt \rfloor}  \la X_{\ell}, {v}_{j}\ra \la \eta_{\ell,3}, \hat{w}_{i,N}- \hat{d}_{i,N} w_i\ra, \\
A_{\lfloor Nt \rfloor}^{(9)}&= \hat{c}_{j,N} \hat{d}_{i,N} \sum_{\ell=1}^{\lfloor Nt \rfloor}  \la X_{\ell},\ v_j\ra \la \eta_{\ell,3}, w_i\ra .\\
\end{aligned}
$$
Theorems \ref{Th:sums} and \ref{Th:eigen} imply that
$$
\begin{aligned}
\sup_{t\in [0,1]} \left|A_{\lfloor Nt \rfloor}^{(7)}\right|
&=\sup_{t\in [0,1]}\left| \Int \Int  \sum_{\ell=1}^{\lfloor Nt \rfloor} X_{\ell}(z) \left(\hat{v}_{j,N}(z) - \hat{c}_{j,N} v_j(z)\right) \hat{w}_{i,N}(x) \sum_{r=1}^q \sum_{n=1}^p \psi_{r,n}\hat{c}_{n,N} w_r(x) \right. \\
&\qquad \times \left.\Int \left( \hat{c}_{n,N} v_n(s) - \hat{v}_{n,N}(s)\right) X_{\ell}(s) \ds \dz \dx \right|\\
&\le \sum_{r=1}^q \sum_{n=1}^p \sup_{t\in [0,1]} \left\|\sum_{\ell=1}^{\lfloor Nt \rfloor} X_{\ell}(s) X_{\ell}(z)\right\| \left\|\hat{v}_{j,N}(z) - \hat{c}_{j,N} v_j(z)\right\|   \\
 &\qquad \times \left\|\hat{w}_{i,N}(x) \psi_{r,n} \hat{c}_{n,N} w_r(x) \right\|    \left\|\hat{c}_{n,N} v_n(s) - \hat{v}_{n,N}(s)\right\|\\
&=\Op{N} \Op{N^{-1/2}} O({1}) \Op{N^{-1/2}},
\end{aligned}
$$
and similarly
$$
\begin{aligned}
\sup_{t\in [0,1]}\left|A_{\lfloor Nt \rfloor}^{(8)}\right|
=\Op{1}.
\end{aligned}
$$
Next we observe that
$$
\begin{aligned}
\sup_{t\in [0,1]}\left|A_{\lfloor Nt \rfloor}^{(9)}\right|
=A_{{\lfloor Nt \rfloor},1}^{(9)}+A_{{\lfloor Nt \rfloor},2}^{(9)},
\end{aligned}
$$
where
$$
\begin{aligned}
A_{{\lfloor Nt \rfloor},1}^{(9)}
&=\hat{c}_{j,N} \hat{d}_{i,N}  {\lfloor Nt \rfloor} \Int \Int \lambda_j v_j(s)w_i(x) \sum_{r=1}^q \sum_{n=1}^p \psi_{r,n} \hat{c}_{n,N} w_r(x) \left(\hat{c}_{n,N} v_n(s) - \hat{v}_{n,N}(s)\right) \ds \dx\\
&={\lfloor Nt \rfloor} R_N^{(3)}(i,j)
\end{aligned}
$$
and
$$
\begin{aligned}
A_{{\lfloor Nt \rfloor},2}^{(9)} &= \hat{c}_{j,N} \hat{d}_{i,N}  \Int \Int \Int \left(\sum_{\ell=1}^{\lfloor Nt \rfloor} X_{\ell}(z)X_{\ell}(s) -{\lfloor Nt \rfloor}C(z,s)\right) v_j(z)w_i(x)\\
&\qquad \times \sum_{r=1}^q \sum_{n=1}^p \psi_{r,n} \hat{c}_{n,N} w_r(x) \left(\hat{c}_{n,N} v_n(s) - \hat{v}_{n,N}(s)\right) \ds \dz\dx.
\end{aligned}
$$
Using  Theorems \ref{Th:eigen} and \ref{Th:menshov} again, we obtain that
$$
\sup_{t\in [0,1]}|A_{{\lfloor Nt \rfloor},2}^{(9)}|=\Op{\log N}.
$$
This  completes the proof.
\end{proof}
\begin{Lemma}\label{l:eta_4} If Assumptions \ref{a:shifts}-\ref{a:PSImoments} hold, then we have
$$
\sup_{t\in[0,1]} \left|\sum_{\ell=1}^{\lfloor Nt \rfloor} \la X_{\ell}, \hat{v}_{j,N}\ra \la \eta_{\ell, 4}, \hat{w}_{i,N}\ra - {\lfloor Nt \rfloor} R_N^{(4)}(i,j)\right|=\Op{\log N}.
$$
\end{Lemma}
\begin{proof} Following the proofs of the previous lemmas we write

$$
\sum_{\ell=1}^{\lfloor Nt \rfloor} \la X_{\ell}, \hat{v}_{j,N}\ra \la\eta_{\ell, 4}, \hat{w}_{i,N}\ra = A_{\lfloor Nt \rfloor}^{(10)} + A_{\lfloor Nt \rfloor}^{(11)} +A_{\lfloor Nt \rfloor}^{(12)} +A_{\lfloor Nt \rfloor}^{(13)},
$$
where
$$
\begin{aligned}
A_{\lfloor Nt \rfloor}^{(10)}&=\sum_{\ell=1}^{\lfloor Nt \rfloor} \la X_{\ell}, \hat{v}_{j,N}- \hat{c}_{j,N} v_j\ra \la \eta_{\ell, 4}, \hat{w}_{i,N}-\hat{d}_{i,N} w_i\ra ,\\
A_{\lfloor Nt \rfloor}^{(11)}&=\hat{c}_{j,N} \hat{d}_{i,N} \sum_{\ell=1}^{\lfloor Nt \rfloor} \la X_{\ell},\ v_j\ra \la \eta_{\ell, 4}, w_i\ra,\\
A_{\lfloor Nt \rfloor}^{(12)}&= \hat{d}_{i,N} \sum_{\ell=1}^{\lfloor Nt \rfloor} \la X_{\ell}, \hat{v}_{j,N}- \hat{c}_{j,N} v_j\ra \la \eta_{\ell, 4}, w_i\ra ,\\
A_{\lfloor Nt \rfloor}^{(13)}&=\hat{c}_{j,N} \sum_{\ell=1}^{\lfloor Nt \rfloor} \la X_{\ell}, v_j\ra \la \eta_{\ell, 4}, \hat{w}_{i,N}-\hat{d}_{i,N} w_i \ra .\\
\end{aligned}
$$
Repeating the arguments used in the proofs of  Lemmas \ref{l:eta_1} and \ref{l:eta_2}, one can show that

$$
\begin{aligned}
\sup_{t\in[0,1]} \left| A_{\lfloor Nt \rfloor}^{(10)}\right|&= \Op{1},\\
\sup_{t\in[0,1]} \left| A_{\lfloor Nt \rfloor}^{(12)}\right|&= \Op{1},\\
\sup_{t\in[0,1]} \left| A_{\lfloor Nt \rfloor}^{(13)}\right|&= \Op{1}.
\end{aligned}
$$
Elementary arguments give
$$
\begin{aligned}
A_{\lfloor Nt \rfloor}^{(11)}
&= A_{{\lfloor Nt \rfloor},1}^{(11)}+A_{{\lfloor Nt \rfloor},2}^{(11)},
\end{aligned}
$$
where
$$
\begin{aligned}
A_{{\lfloor Nt \rfloor},2}^{(11)}&=\hat{c}_{j,N} \hat{d}_{i,N} \Int \Int \Int \left(\sum_{\ell=1}^{\lfloor Nt \rfloor} X_{\ell}(z) X_{\ell}(s) -{\lfloor Nt \rfloor}C(z,s) \right)v_j(z)  w_i(x)  \\
&\qquad \times \sum_{r=1}^{q} \sum_{n=1}^{p} \hat{d}_{r,N} \psi_{r,n} \hat{c}_{n,N} \left(\hat{d}_{r,N} w_r(x)-\hat{w}_{r,N}(x)\right)  \hat{c}_{n,N} v_n(s) \ds \dz \dx,
\end{aligned}
$$
and
$$
\begin{aligned}
A_{{\lfloor Nt \rfloor},1}^{(11)}
&={\lfloor Nt \rfloor} \hat{c}_{j,N} \hat{d}_{i,N} \lambda_j \Int    w_i(x) \sum_{r=1}^{q} \hat{d}_{r,N} \psi_{r,j} \hat{c}_{j,N} \left(\hat{d}_{r,N} w_r(x)-\hat{w}_{r,N}(x)\right)  \hat{c}_{j,N} \dx \\
&={\lfloor Nt \rfloor} R_N^{(4)}(i,j).
\end{aligned}
$$

Using  Theorems \ref{Th:eigen} and \ref{Th:menshov} again, we conclude that

$$
\sup_{t\in[0,1]} \left| A_{{\lfloor Nt \rfloor},2}^{(11)}\right|=\Op{\log N},\\
$$
completing the proof.
\end{proof}
\begin{Lemma}\label{l:eta_5} If Assumptions \ref{a:shifts}-\ref{a:eigenu} hold, then we have
$$
\sup_{t\in[0,1]} \left| \sum_{\ell=1}^{\lfloor Nt \rfloor} \la X_{\ell}, \hat{v}_{j,N}\ra \la\eta_{\ell,5}, \hat{w}_{i,N}\ra \right| = \Op{1}.
$$
\end{Lemma}
\begin{proof} It follows from Theorems \ref{Th:sums} and \ref{Th:eigen} that
$$
\begin{aligned}
&\sup_{t\in[0,1]}\left| \sum_{\ell=1}^{\lfloor Nt \rfloor} \la X_{\ell}, \hat{v}_{j,N}\ra \la \eta_{\ell,5}, \hat{w}_{i,N}\ra  \right|\\
&=\sup_{t\in[0,1]}\left| \sum_{\ell=1}^{\lfloor Nt \rfloor} \Int \Int X_{\ell}(z) \hat{v}_{j,N}(z)\hat{w}_{i,N}(x) \sum_{r=1}^{q} \sum_{n=1}^{p} \hat{d}_{r,N} \psi_{r,n} \hat{c}_{n,N} \left(\hat{w}_{r,N}(x)- \hat{d}_{r,N} w_r(x)\right) \right. \\
&\qquad \times \left.\Int\left( \hat{c}_{n,N} v_n(s)-\hat{v}_{n,N}(s)\right) X_{\ell}(s)\ds \dz \dx \right|\\
&\le \sum_{r=1}^{q} \sum_{n=1}^{p} |\psi_{r,n}| \sup_{t\in[0,1]}\left\| \sum_{\ell=1}^{\lfloor Nt \rfloor}  X_{\ell}(z)  X_{\ell}(s) \right\|  \left\| \hat{w}_{i,N}(x)
\right\| \left\|\hat{v}_{j,N}(z) \right\|  \left\|\hat{w}_{r,N}(x)- \hat{d}_{r,N} w_r(x)\right\| \\
&\qquad \times \left\|\hat{c}_{n,N} v_n(s)-\hat{v}_{n,N}(s)\right\|\\
&=\Op{N}\Op{N^{-1/2}}\Op{N^{-1/2}}.
\end{aligned}
$$
\end{proof}
\begin{Lemma}\label{l:combining} If Assumptions \ref{a:shifts}-\ref{a:PSImoments} hold, then we have
$$
\sup_{t\in[0,1]} \left|\sum_{\ell=1}^{\lfloor Nt \rfloor} \la X_{\ell}, \hat{v}_{j,N}\ra \la \epsilon_{\ell}^{**}, \hat{w}_{i,N}\ra  - \left( {\lfloor Nt \rfloor} R_N(i,j) + \hat{c}_{j,N}\hat{d}_{i,N} \sum_{\ell =1}^{\lfloor Nt \rfloor} \gamma_{\ell}(i,j)\right)\right| = \Op{\log N}.
$$
\end{Lemma}
\begin{proof}

Combining Lemmas \ref{l:epsilon} - \ref{l:eta_5}, we immediately see that
$$
\sup_{t\in[0,1]} \left|\sum_{\ell=1}^{\lfloor Nt \rfloor} \la X_{\ell}, \hat{v}_{j,N}\ra \la \epsilon_{\ell}^{**}, \hat{w}_{i,N}\ra  - \left({\lfloor Nt \rfloor} R_N(i,j)+\hat{c}_{j,N}\hat{d}_{i,N} T_{\lfloor Nt \rfloor}(i,j) \right)\right| = \Op{\log N},
$$
where
$$
T_{\lfloor Nt \rfloor}(i,j)=T_{\lfloor Nt \rfloor}^{(1)}(i,j)+T_{\lfloor Nt \rfloor}^{(2)}(i,j).
$$
Thus we need only to show that
$$
T_{\lfloor Nt \rfloor}^{(2)}(i,j) = \sum_{\ell =1}^{\lfloor Nt \rfloor}\la X_\ell, v_j\ra \la X_\ell, u_i\ra.
$$
However, using that the $v_j$'s are orthogonal eigenfunctions of $C$, we get that
$$
\begin{aligned}
\Int \Int C(z,s) v_j(z) \sum_{r=p+1}^{\infty} \psi_{i,r}v_r(s) \ds \dz=\lambda_j
\Int v_j(s) \sum_{r=p+1}^{\infty} \psi_{i,r}v_r(s) \ds \dz
=0,
\end{aligned}
$$
completing the proof.
\end{proof}

\begin{Lemma}\label{l:betahat} If Assumptions \ref{a:shifts}-\ref{a:PSImoments} hold, then we have
$$
\left|\sqrt{N}(\boldsymbol \beta-\hat{\boldsymbol \beta}_N)\right|=\Op{1}.
$$
\end{Lemma}
\begin{proof} It is easy to see that
$$
\sqrt{N}(\boldsymbol \beta-\hat{\boldsymbol \beta}_N)= -N^{-1/2}\left(\frac{\hat{\mathbf Z}_N^T\hat{\mathbf Z}_N}{N}\right)^{-1}\hat{\mathbf Z}_N^T\hat{\boldsymbol \Delta}_N.
$$
It follows from (\ref{e:prod-2}) that
$$
\left|\left(\frac{\hat{\mathbf Z}_N^T\hat{\mathbf Z}_N}{N}\right)^{-1}\right|=\Op{1}.
$$
Lemma \ref{l:combining} and \eqref{eq:zzz} yield that
$$
\left|\hat{\mathbf Z}_N^T\hat{\boldsymbol \Delta}_N\right|\leq \max_{1\leq i\leq q, 1\leq j \leq p}\left\{N|R_N(i,j)| +\left| \sum_{\ell=1}^N\gamma_\ell(i,j)\right| \right\}+\Op{\log N}.
$$
It follows from  Theorem \ref{Th:eigen} that for all $1\leq i\leq q, 1\leq j \leq p$
$$
N|R_N(i,j)|=\Op{N^{1/2}}
$$
while Theorem \ref{aueapp} implies that
$$
\left| \sum_{\ell=1}^N\gamma_\ell(i,j)\right|=\Op{N^{1/2}}.
$$
\end{proof}
\begin{Lemma}\label{l:just-sums} If Assumptions \ref{a:shifts}-\ref{a:PSImoments} hold, then we have
$$
\sup_{t\in [0,1]}\left|\tilde{\bV}_N(t)-{\boldsymbol \zeta}_N \frac{1}{N^{1/2}}\left(\sum_{\ell=1}^{\lfloor Nt \rfloor}{\boldsymbol \gamma}_{\ell} - t\sum_{\ell=1}^{N}{\boldsymbol \gamma}_{\ell}\right) \right|=\op{1}.
$$
\end{Lemma}
\begin{proof} Lemmas \ref{l:convergetozero} and \ref{l:betahat} and \eqref{eq:star} imply that
$$
\sup_{t\in [0,1]}\biggl|\tilde{\bV}_N(t)-\frac{1}{N^{1/2}}\left(\hat{\mathbf Z}_{\lfloor Nt \rfloor}^T \hat{\boldsymbol \Delta}_{\lfloor Nt \rfloor} - t\hat{\mathbf Z}_{N}^T \hat{\boldsymbol \Delta}_{N}\right)\biggl|=\op{1}.
$$
It also follows from Lemma \ref{l:combining} and \eqref{eq:zzz}
$$
\sup_{t\in [0,1]}\biggl|\hat{\mathbf Z}_{\lfloor Nt \rfloor}^T \hat{\boldsymbol \Delta}_{\lfloor Nt \rfloor}-\mbox{vec}\biggl(\biggl\{ {\lfloor Nt \rfloor} R_N(i,j) + \hat{c}_{j,N}\hat{d}_{i,N} \sum_{\ell =1}^{\lfloor Nt \rfloor} \gamma_{\ell}(i,j)\biggl\}^T\biggl)\biggl| = \Op{\log N},
$$
and therefore the proof is complete.
\end{proof}

Now we have all the necessary tools to prove the main result.
\begin{proof}[{\bf Proof of Theorem \ref{Th:main}}] It follows from Lemma \ref{l:just-sums} and Theorem \ref{aueapp} that
$$
{\boldsymbol \zeta}_N\tilde{\bV}_N(t)\;\;\;\stackrel{{\cal D}^{pq}[0,1]}{\longrightarrow}\;\;\; \bW_{{\boldsymbol \Sigma}}(t)-t\bW_{{\boldsymbol \Sigma}}(1).
$$
Next we observe that
$$
\left\{{\boldsymbol \Sigma}^{-1/2}(\bW_{{\boldsymbol \Sigma}}(t)-t\bW_{{\boldsymbol \Sigma}}(1)), 0\leq t\leq 1\right\} \stackrel{{\cal D}}{=}\{{\boldsymbol B}(t),0\leq t\leq 1\},
$$
where ${\boldsymbol B}(t)=({\mathcal B}_1(t), \ldots {\mathcal B}_{pq}(t))^T$ and  ${\mathcal B}_1, \ldots {\mathcal B}_{pq}$ are independent, identically distributed Brownian bridges. Hence
$$
({\boldsymbol \zeta}_N\tilde{\bV}_N(t))^T{\boldsymbol \Sigma}^{-1}({\boldsymbol \zeta}_N\tilde{\bV}_N(t))\;\;\;\stackrel{{\cal D}[0,1]}{\longrightarrow}\;\;\;\sum_{\ell=1}^{pq}{\mathcal B}_\ell^2(t).
$$
Now, using Assumption \ref{a:consistent} with Slutsky's lemma, the proof is complete.
\end{proof}


\section{Proof of Theorems \ref{Th:bart-1} and \ref{Th:bart-2}}\label{s:dog}
We can assume without loss of generality that $K(u)=0$ if $|u|>1$.
Let $m$ be a positive integer and define
$$
\bgamma_\ell^{(m)}=\mbox{vec}(\{\gamma_\ell^{(m)}(i,j), 1\leq i \leq q, 1\leq j \leq p\}^T),
$$
where
$$
\gamma_\ell^{(m)}(i,j)=\la X_\ell^{(m)}, v_j\ra \la \epsilon_\ell^{(m)}, w_i\ra+
\la X_\ell^{(m)}, v_j\ra \la X_\ell^{(m)}, u_i\ra .
$$
The long term covariance matrix associated with the stationary sequence $\{\bgamma_\ell^{(m)}, 1\leq \ell <\infty\}$ is given by
$$
 {\boldsymbol \Sigma}^{(m)}=\E{\bgamma}_1^{(m)}({\bgamma}_1^{(m)})^T + \sum_{\ell=1}^{\infty} \E{\bgamma}_1^{(m)}( {\bgamma}_{\ell+1}^{(m)})^T + \sum_{\ell=1}^{\infty} \E {\bgamma}_{\ell+1}^{(m)}( {\bgamma}_1^{(m)})^T.
$$
The corresponding Bartlett estimator is defined as
$$
\tilde{\boldsymbol \Sigma}_N^{(m)} = \sum_{k=-(N-1)}^{N-1} K(k/B_N) {\boldsymbol \phi}_{k,N}^{(m)},
$$
where
$$
{\boldsymbol \phi}_{k,N}^{(m)}= \frac{1}{N} \sum_{\ell=\max(1,1-k)}^{\min(N,N-k)} {\bgamma}_{\ell}^{(m)}({\bgamma}_{\ell+k}^{(m)})^T
$$
are the sample covariances of lag $k$. Since $K$ is symmetric, $K(0)=1$  and $K(u)=0$ outside $[-1,1]$ we have that
$$
\tilde{\boldsymbol \Sigma}_N^{(m)} = {\boldsymbol \phi}_{0,N}^{(m)} +\sum_{k=1}^{B_N} K(k/B_N) {\boldsymbol \phi}_{k,N}^{(m)}+\sum_{k=1}^{B_N} K(k/B_N)( {\boldsymbol \phi}_{k,N}^{(m)})^T
$$
for all sufficiently large $N$.\\

 We start with the consistency of $\tilde{\boldsymbol \Sigma}_N^{(m)}$.
\begin{Lemma}\label{l:m-dep} If Assumptions 
\ref{a:K} and
\ref{a:kernel} are satisfied, then we have for every $m$
$$
\tilde{\boldsymbol \Sigma}_N^{(m)}\convP {\boldsymbol \Sigma}^{(m)},
$$
as $N\rightarrow \infty.$
\end{Lemma}
\begin{proof}
Since the sequence ${\boldsymbol \gamma}_\ell^{(m)}$ is $m$-dependent we have that
$$
{\boldsymbol \Sigma}^{(m)}=\E{\boldsymbol \gamma}_1{\boldsymbol \gamma}_1^T + \sum_{\ell=1}^{m} \E{\boldsymbol \gamma}_1 {\boldsymbol \gamma}_{\ell+1}^T + \sum_{\ell=1}^{m} \E {\boldsymbol \gamma}_{\ell+1} {\boldsymbol \gamma}_1^T.
$$
It follows from the ergodic theorem that for any fixed $k$ and $m$
$$
{\boldsymbol \phi}_{k,N}^{(m)}\convP \E{\bgamma}_{1}^{(m)}({\bgamma}_{1+k}^{(m)})^T.
$$
So using Assumptions \ref{a:K}(i),  \ref{a:K}(ii) and \ref{a:kernel}  we get that
$$
{\boldsymbol \phi}_{0,N}^{(m)} +\sum_{k=1}^{m} K(k/B_N) {\boldsymbol \phi}_{k,N}^{(m)}+\sum_{k=1}^{m} K(k/B_N)( {\boldsymbol \phi}_{k,N}^{(m)})^T\convP
\E{\boldsymbol \gamma}_1{\boldsymbol \gamma}_1^T + \sum_{\ell=1}^{m} \E{\boldsymbol \gamma}_1 {\boldsymbol \gamma}_{\ell+1}^T + \sum_{\ell=1}^{m} \E {\boldsymbol \gamma}_{\ell+1} {\boldsymbol \gamma}_1^T.
$$
Lemma \ref{l:m-dep} is proven if we show that
\begin{equation}\label{eq:zero-1}
\sum_{k=m+1}^{B_N} K(k/B_N) {\boldsymbol \phi}_{k,N}^{(m)}\convP 0
\end{equation}
and
\begin{equation}\label{eq:zero-2}
\sum_{k=m+1}^{B_N} K(k/B_N)( {\boldsymbol \phi}_{k,N}^{(m)})^T\convP 0.
\end{equation}
Clearly, it is enough to prove (\ref{eq:zero-1}).\\

Let
$$
{\mathbf G}_{N}^{(m)}=\sum_{k=m+1}^{B_N} K(k/B_N) {\boldsymbol \phi}_{k,N}^{(m)}.
$$
Elementary arguments show that
$$
\begin{aligned}
{\mathbf G}_{N}^{(m)}&=\sum_{k=m+1}^{B_N} K(k/B_N) {\boldsymbol \phi}_{k,N}^{(m)}\\
&= \sum_{k=m+1}^{B_N} K(k/B_N) \frac{1}{N} \sum_{\ell=1}^{N-k} {\boldsymbol \gamma}_{\ell}^{(m)} \left({\boldsymbol \gamma}_{\ell+k}^{(m)}\right)^T\\
&= \sum_{\ell=1}^{N-(m+1)}{\boldsymbol \gamma}_{\ell}^{(m)} {\mathbf H}_{\ell,N}^{(m)},
\end{aligned}
$$

where

$$
{\mathbf H}_{\ell,N}^{(m)}=\sum_{k=m+1}^{\min(N-\ell, B_N)}  \frac{K(k/B_N)}{N}  \left({\boldsymbol \gamma}_{\ell+k}^{(m)}\right)^T.
$$
Let
$$
{ G}_{N}^{(m)}(i,j)
= \sum_{\ell=1}^{N-(m+1)}{ \gamma}_{\ell}^{(m)}(i) { H}_{\ell,N}^{(m)}(j), \;\;\;1\leq i, j \leq pq,
$$

where ${ \gamma}_{\ell}^{(m)}(i)$ and $ { H}_{\ell,N}^{(m)}(j)$ are the
$i^{\mbox{\it {th}}}$ and the $j^{\mbox{{\it th}}}$ coordinates of the vectors ${\boldsymbol \gamma}_{\ell,N}^{(m)}$ and ${\mathbf H}_{\ell,N}^{(m)}$, respectively.
Next  we write
$$
\begin{aligned}
\E\left( G_{N}^{(m)}(i,j)\right)^2&=\E\left(\sum_{\ell=1}^{N-(m+1)}{\gamma}_{\ell}^{(m)}(i) { H}_{\ell,N}^{(m)}(j)\right)^2\\
&=\mathop{\mathop{\sum \sum}_{1\le r \le N-(m+1)}}_{1\le \ell \le N-(m+1)} \E\left({ H}_{\ell,N}^{(m)}(j){ \gamma}_{\ell}^{(m)}(i) {\gamma}_{r}^{(m)} (i){ H}_{r,N}^{(m)}(j)\right)\\
&= { G}_{1,N}^{(m)}(i,j) + { G}_{2,N}^{(m)}(i,j),
\end{aligned}
$$
where
$$
{ G}_{1,N}^{(m)}(i,j)=\mathop{\mathop{\mathop{\sum \sum}_{1\le r \le N-(m+1)}}_{1\le \ell \le N-(m+1)}}_{|r-\ell|\le m} \E\left({ H}_{\ell,N}^{(m)}(j) { \gamma}_{\ell}^{(m)}(i) { \gamma}_{r}^{(m)}(i) { H}_{r,N}^{(m)}(j)\right),
$$
and
$$
\begin{aligned}
{ G}_{2,N}^{(m)}(i,j)&=\mathop{\mathop{\mathop{\sum \sum}_{1\le r \le N-(m+1)}}_{1\le \ell \le N-(m+1)}}_{|r-\ell| > m} \E\left({ H}_{\ell,N}^{(m)}(j){ \gamma}_{\ell}^{(m)}(i) { \gamma}_{r}^{(m)}(i) { H}_{r,N}^{(m)}(j)\right).
\end{aligned}
$$
Notice that   ${\boldsymbol \gamma}_{\ell}^{(m)}$ is independent of ${\mathbf H}_{\ell,N}^{(m)}$, ${\mathbf H}_{r,N}^{(m)}$ and ${\boldsymbol \gamma}_{r}^{(m)}$, if $r>m+\ell$. Hence
$$
\begin{aligned}
\E\left({ H}_{\ell,N}^{(m)}(j) { \gamma}_{\ell}^{(m)}(i) \gamma_{r}^{(m)}(i) { H}_{r,N}^{(m)}(j)\right)&=\begin{cases}
\E{ \gamma}_{\ell}^{(m)}(i)\E\left({ H}_{\ell,N}^{(m)} (j)  \gamma_{r}^{(m)}(i) { H}_{r,N}^{(m)}(j)\right) & r>m+\ell,\\
\E\gamma_{r}^{(m)}(i)\E\left( H_{\ell,N}^{(m)}(j) \gamma_{\ell}^{(m)}(i)
 { H}_{r,N}^{(m)}(j)\right) & \ell > m+r,\\
\E\left( H_{\ell,N}^{(m)}(j) \gamma_{\ell}^{(m)}(i) \gamma_{r}^{(m)} (i){ H}_{r,N}^{(m)}(j)\right) & |\ell -r|\le m,\\
\end{cases}\\[.3cm]
&=\begin{cases}
0 & |\ell -r|> m,\\
\E\left( H_{\ell,N}^{(m)}(j) \gamma_{\ell}^{(m)}(i)\gamma_{r}^{(m)}(i) { H}_{r,N}^{(m)}(j)\right) & |\ell -r|\le m.\\
\end{cases}
\end{aligned}
$$
Thus we have
$$
\E{ G}_{2,N}^{(m)}(i,j)=0.
$$
Let $M$ be an upper bound on $|K(t)|$.  Using the fact that $ {\boldmath \gamma}_{\ell}^{(m)}$ is an m-dependent sequence, we now obtain the following:

\begin{align}\label{e:bound}
\E( H_{\ell,N}^{(m)}(j))^2&= \sum_{k=m+1}^{\min(N-\ell, B_N)}\sum_{v=m+1}^{\min(N-\ell, B_N)}  \frac{K(k/B_N)}{N} \frac{K(v/B_N)}{N} \E\left( \gamma_{\ell+k}^{(m)}(j) \gamma_{\ell+v}^{(m)}(j)\right)\\
&\le \frac{M^2}{N^2}\sum_{k=m+1}^{\min(N-\ell, B_N)}\sum_{v=m+1}^{\min(N-\ell, B_N)}  \E\left( \gamma_{\ell+k}^{(m)} (j)\gamma_{\ell+v}^{(m)}(j)\right) \notag\\
&\le \frac{M^2}{N^2} B_N  \sum_{r=-m}^m \E\left| \gamma_{0}^{(m)}(j) \gamma_{r}^{(m)}(j)\right|\notag\\
&=O\left(\frac{B_N}{N^2}\right).\notag
\end{align}
In the next step we will first  use the Cauchy-Schwarz inequality, then the independence of
$H_{\ell,N}^{(m)}(j)$ and $\gamma_{\ell}^{(m)}(i)$ and the independence of $H_{r,N}^{(m)}(j)$ and $\gamma_{r}^{(m)}(i)$ to get
$$
\begin{aligned}
\left|{ G}_{2,N}^{(m)}(i,j)\right|&
\le \mathop{\mathop{\mathop{\sum \sum}_{1\le r \le N-(m+1)}}_{1\le \ell \le N-(m+1)}}_{|r-\ell|\le m} \E\left| H_{\ell,N}^{(m)}(j)  \gamma_{\ell}^{(m)}(i) \gamma_{r}^{(m)}(i) { H}_{r,N}^{(m)}(j)\right|\\
&\le \mathop{\mathop{\mathop{\sum \sum}_{1\le r \le N-(m+1)}}_{1\le \ell \le N-(m+1)}}_{|r-\ell|\le m} \left(\E\left( H_{\ell,N}^{(m)}(j) \gamma_{\ell}^{(m)}(i)\right)^2\right)^{1/2}\left(\E\left(  \gamma_{r}^{(m)}(i) { H}_{r,N}^{(m)}(j)\right)^2\right)^{1/2}\\
&\le \mathop{\mathop{\mathop{\sum \sum}_{1\le r \le N-(m+1)}}_{1\le \ell \le N-(m+1)}}_{|r-\ell|\le m} \left(\E\left({ H}_{\ell,N}^{(m)}(j) \right)^2\right)^{1/2} \left(\E\left( \gamma_{\ell}^{(m)}(i)\right)^2\right)^{1/2}\left(\E\left(  \gamma_{r}^{(m)}(i)\right)^2\right)^{1/2} \\
&\hspace{2 cm}\times \left(\E\left( { H}_{r,N}^{(m)}(j)\right)^2\right)^{1/2}\\
&\le 2mN O\left(\frac{{B_N^{1/2}}}{N}\right) O(1) O(1) O\left(\frac{{B_N^{1/2}}}{N}\right)\\
&=  O\left(\frac{B_N}{N}\right)\\
&=o(1),
\end{aligned}
$$
where we also used (\ref{e:bound}) and Assumption \ref{a:kernel}. This  completes the proof of Lemma \ref{l:m-dep}.
\end{proof}

\medskip


Let  $\im^2=-1$.
\begin{Lemma}\label{l:m-ssum} If Assumptions \ref{a:shifts}-\ref{a:k-dependent}, \ref{a:K} and
\ref{a:kernel} are satisfied, then for all $1\leq j \leq pq$ we have
\begin{equation}\label{eq:ssum-1}
\limsup_{N\rightarrow \infty}\limsup_{m\rightarrow \infty}\sup_{-\infty<t<\infty}\E\left(\frac{1}{N^{1/2}} \sum_{k=1}^N(\gamma_k(j)- \gamma_k^{(m)}(j))e^{\im kt}\right)^2 =0,
\end{equation}
\begin{equation}\label{eq:ssum-2}
\limsup_{N\rightarrow \infty}\limsup_{m\rightarrow \infty}\sup_{-\infty<t<\infty}\E\left(\frac{1}{N^{1/2}} \sum_{k=1}^N\gamma_k(j)e^{\im kt}\right)^2 <\infty
\end{equation}
and
\begin{equation}\label{eq:ssum-3}
\limsup_{N\rightarrow \infty}\limsup_{m\rightarrow \infty}\sup_{-\infty<t<\infty}\E\left(\frac{1}{N^{1/2}} \sum_{k=1}^N \gamma_k^{(m)}(j)e^{\im kt}\right)^2 <\infty.
\end{equation}
\end{Lemma}
\begin{proof}
First we note that
\begin{align}\notag
\E\left( \sum_{k=1}^N(\gamma_k(j)- \gamma_k^{(m)}(j))e^{\im kt}\right)^2=&
\sum_{1\leq k\leq N}\E ((\gamma_k(j)- \gamma_k^{(m)}(j))e^{\im kt})^2 \notag
\\
&+2\sum_{1\leq k < \ell\leq N}\E\left[ (\gamma_k(j)- \gamma_k^{(m)}(j))(\gamma_\ell(j)- \gamma_\ell^{(m)}(j))\right]e^{\im (k+\ell)t}.\notag
\end{align}
It follows from (\ref{eq:aue}) that there is a sequence $
c_1(m)\rightarrow 0$ such that
$$
\left|\sum_{1\leq k\leq N}\E (\gamma_k(j)- \gamma_k^{(m)}(j))^2e^{\im 2kt}\right| \leq Nc_1(m).
$$
Next we write
$$
\begin{aligned}
\sum_{1\leq k < \ell\leq N}&\E\left[ (\gamma_k(j)- \gamma_k^{(m)}(j))\gamma_\ell(j)\right]e^{\im (k+\ell)t}\\
&=\sum_{1\leq k < \ell\leq N}\E\left[ (\gamma_k(j)- \gamma_k^{(m)}(j))(\gamma_\ell(j)-\gamma_\ell^{(\ell-k)}(j))\right]e^{\im (k+\ell)t},
\end{aligned}
$$
since $(\bgamma_k,   \bgamma_k^{(m)})$ and $\bgamma_\ell^{(\ell-k)}$ are independent. Using the Cauchy-Schwarz inequality first, then   (\ref{eq:aue})
again we get that
$$
\begin{aligned}
\sum_{1\leq k < \ell\leq N}&\left|\E\left[ (\gamma_k(j)- \gamma_k^{(m)}(j))(\gamma_\ell(j)-\gamma_\ell^{(\ell-k)}(j))\right]e^{\im (k+\ell)t}\right|\\
&\leq \sum_{1\leq k < \ell\leq N}\left[\E(\gamma_k(j)- \gamma_k^{(m)}(j))^2\right]^{1/2}\left[\E(\gamma_\ell(j)-\gamma_\ell^{(\ell-k)}(j))^2\right]^{1/2}\\
&\leq N\left[\E(\gamma_1(j)- \gamma_1^{(m)}(j))^2\right]^{1/2}\sum_{1\leq k < \infty}
\left[\E(\gamma_1(j)-\gamma_1^{(k)}(j))^2\right]^{1/2}\\
&=Nc_2(m)
\end{aligned}
$$
with some sequence $c_2(m)\rightarrow 0$. Similar arguments show that
$$
\sum_{1\leq k < \ell\leq N}\left|\E\left[ (\gamma_k(j)- \gamma_k^{(m)}(j))\gamma_\ell^{(m)}(j)\right]e^{\im (k+\ell)t}\right|= N c_3(m)
$$
with some  sequence $c_3(m)\rightarrow 0$, completing the proof of (\ref{eq:ssum-1}).\\

Similarly to the proof of (\ref{eq:ssum-1}), we write
$$
\begin{aligned}
\E\left( \sum_{k=1}^N\gamma_k(j)e^{\im kt}\right)^2&=
\sum_{k=1}^N\sum_{\ell=1}^N\E \gamma_k(j)\gamma_\ell(j)e^{\im (k+\ell)t}\\
&=\sum_{k=1}^N\E \gamma_k^2(j)e^{2\im kt}+2\sum_{1\leq k<\ell \leq N}\E \gamma_k(j)\gamma_\ell(j)e^{\im (k+\ell)t}\\
&=\E \gamma_1^2(j)\sum_{k=1}^Ne^{2\im kt}+2\sum_{1\leq k<\ell \leq N}\E \gamma_k(j)(\gamma_{\ell}(j)-\gamma_\ell^{(\ell-k)}(j))e^{\im (k+\ell)t},
\end{aligned}
$$
since by the independence of $\gamma_k(j)$ and $\gamma_\ell^{(\ell-k)}(j)$ we have that $\E\gamma_k(j)\gamma_\ell^{(\ell-k)}(j)=0$. Using the Cauchy-Schwarz inequality with (\ref{eq:aue}) we get that
$$
\left|\sum_{1\leq k<\ell \leq  N}\E \gamma_k(j)(\gamma_{\ell}(j)-\gamma_\ell^{(\ell-k)}(j))e^{\im (k+\ell)t}\right|\leq cN
$$
with some constant $c$, completing the proof of (\ref{eq:ssum-2}). The same  arguments can be used to prove (\ref{eq:ssum-3}).
\end{proof}

\medskip
Following Taniguchi and Kakizawa  (2000) we define ${\mathbf S}_N(t)=\sum_{k=1}^N {\boldsymbol \gamma}_{k,N} e^{\im kt}$ and ${\mathbf S}_N^{(m)}(t)=\sum_{k=1}^N {\boldsymbol \gamma}_{k,N}^{(m)} e^{\im kt}$.  Let ${\mathbf S}_N^*(t)$ be the conjugate transpose of ${\mathbf S}_N(t)$ and introduce
$$
\begin{aligned}
{\mathbf I}_N(t)&=\frac{1}{N} {\mathbf S}_N(t){\mathbf S}_N^*(t)\\
&=\frac{1}{N} \sum_{k=1}^N {\boldsymbol \gamma}_{k} e^{\im kt}    \sum_{{\ell}=1}^N {\boldsymbol \gamma}_{\ell}^T e^{-\im {\ell}t}\\
&=\frac{1}{N} \sum_{{\ell}=1}^N \sum_{k=1}^N e^{\im t(k-{\ell})} {\boldsymbol \gamma}_{k}     {\boldsymbol \gamma}_{\ell}^T \\
&=\sum_{k=1-N}^{N-1} e^{-\im tk}  \frac{1}{N} \sum_{\ell=\max(1,1-k)}^{\min(N,N-k)} {\boldsymbol \gamma}_{k}{\boldsymbol \gamma}_{\ell+k}^T\\
&=\sum_{k=1-N}^{N-1} e^{-\im tk}  {\boldsymbol \phi}_{k,N}.\\
\end{aligned}
$$
Similarly we define
$$
\begin{aligned}
{\mathbf I}_N^{(m)}(t)&=\frac{1}{N} {\mathbf S}^{(m)}_N(t)\left({\mathbf S}^{(m)}_N(t)\right)^*
=\sum_{k=1-N}^{N-1} e^{-\im tk}  {\boldsymbol \phi}_{k,N}^{(m)}.\\
\end{aligned}
$$

\begin{Lemma}\label{l:justwhatweneeded} If Assumptions \ref{a:shifts}-\ref{a:k-dependent}, \ref{a:K} and
\ref{a:kernel} are satisfied, then we have
$$
\limsup_{N\rightarrow \infty}\limsup_{m\rightarrow \infty} \sup_{-\infty <t <\infty}\E\left| {\mathbf I}_{N}({t}) - {\mathbf I}_{N}^{(m)}({t})\right| =0.
$$
\end{Lemma}
\begin{proof} By the triangle inequality we have
$$
\begin{aligned}
 \left| {\mathbf I}_N({t}) - {\mathbf I}_N^{(m)}({t})\right|
&=
\left| \frac{1}{N}{\mathbf S}_N({t}){\mathbf S}_N^*({t}) - \frac{1}{N}{\mathbf S}_N^{(m)}({t})\left({\mathbf S}_N^{(m)}({t})\right)^* \right|
\\
&\leq  \frac{1}{N}\left|{\mathbf S}_N({t})({\mathbf S}_N^*({t}) -({\mathbf S}_N^{(m)}({t}))^*) \right|\\
&\qquad + \frac{1}{N}\left|({\mathbf S}_N({t})-{\mathbf S}_N^{(m)}({t}))( {\mathbf S}_N^{(m)}({t}))^*\right|.
\end{aligned}
$$
Now the result follows from Lemma \ref{l:m-ssum} via the Cauchy-Schwartz inequality.
\end{proof}

\begin{proof}[{\bf Proof of Theorem \ref{Th:bart-1}}]
Define the Fourier transform, $\hat{K}(u)$, of the kernel $K$ as $\hat{K}(u) = \frac{1}{2\pi} \int_{-\infty}^{\infty} K(s) e^{-\im su}\ds$.  Since $K$ and $\hat{K}$ are in $L^1$ and both are Lipschitz functions, the inversion formula gives $K(s)= \int_{-\infty}^{\infty} \hat{K}(u) e^{\im su}\du.$  From the relationship between $K$ and $\hat{K}$ and from the fact that $K$ is supported on the interval $[-1,1]$, we obtain:
$$
\begin{aligned}
\tilde{\boldsymbol \Sigma}_N &= \sum_{k=-B_N}^{B_N} K(k/B_N) {\boldsymbol \phi}_{k,N}\\
&= \sum_{k=1-N}^{N-1} K(k/B_N) {\boldsymbol \phi}_{k,N}\\
&= \sum_{k=1-N}^{N-1} \left( \int_{-\infty}^{\infty} \hat{K}(u) e^{\im (k/B_N)u}\du\right) {\boldsymbol \phi}_{k,N}\\
&= \int_{-\infty}^{\infty} \hat{K}(u) \sum_{k=1-N}^{N-1} e^{-\im(-u/B_N)k} {\boldsymbol \phi}_{k,N} \du\\
&= \int_{-\infty}^{\infty} \hat{K}(u) {\mathbf I}_N(-u/B_N) \du.\\
\end{aligned}
$$
Similarly,
$$
\tilde{\boldsymbol \Sigma}_N^{(m)} = \int_{-\infty}^{\infty} \hat{K}(u) {\mathbf I}_N^{(m)}(-u/B_N) \du.
$$
Hence we have
$$
\begin{aligned}
 \E \biggl|\tilde{\boldsymbol \Sigma}_N - \tilde{\boldsymbol \Sigma}_N^{(m)}
\biggl|
&= \E \left|\int_{-\infty}^{\infty} \hat{K}(u) \left({\mathbf I}_N(u/B_N) - {\mathbf I}_N^{(m)}(u/B_N)\right) \du \right|\\
&\le  \int_{-\infty}^{\infty}\left|\hat{K}(u)\right| \E\left|  \left({\mathbf I}_N(u/B_N) - {\mathbf I}_N^{(m)}(u/B_N)\right) \right| \du \\
&\le  \sup_{-\infty <t <\infty} \left|\left| {\mathbf I}_N(t) - {\mathbf I}_N^{(m)}(t)\right| \right|_1 \int_{-\infty}^{\infty}\left|\hat{K}(u)\right|  \du .
\end{aligned}
$$
Applying Lemma \ref{l:justwhatweneeded} we conclude that
$$
\biggl|\tilde{\boldsymbol \Sigma}_N - \tilde{\boldsymbol \Sigma}_N^{(m)}
\biggl|\convP 0,
$$
as $\min(N,m)\rightarrow \infty$. On the other hand, by Lemma \ref{l:m-dep},     for every fixed $m$
$$
\tilde{\boldsymbol \Sigma}_N^{(m)}\convP {\boldsymbol \Sigma}^{(m)}.
$$
Since
$$
{\boldsymbol \Sigma}^{(m)}\rightarrow {\boldsymbol \Sigma},
$$
as $m\rightarrow \infty$, the proof of the theorem is complete.
\end{proof}
\begin{proof}[{\bf Proof of Theorem \ref{Th:bart-2}}] It follows from the definition of $\hat{\epsilon}_\ell$, (\ref{eq:Y}) and the orthonormality of $\{w_j, 1\leq j <\infty\}$ that
$$
\begin{aligned}
\la \hat{\epsilon}_\ell, w_i\ra =\la \epsilon_\ell, w_i\ra +\la X_\ell, u_i\ra +\la \nu_\ell, w_i\ra,
\end{aligned}
$$
where
$$
 \nu_\ell(t)=\sum_{i=1}^q\sum_{j=1}^p{\psi}_{i,j}{w}_{i}(t)\la X_\ell, {v}_{j}\ra -\sum_{i=1}^q\sum_{j=1}^p\hat{\psi}_{i,j}\hat{w}_{i,N}(t)\la X_\ell, \hat{v}_{j,N}\ra.
$$
Following the proof of Theorem  \ref{Th:bart-1} one can show that the estimates in (\ref{e:eig-w}) and (\ref{e:eig-v}) yield
\begin{equation}\label{eq:who}
\left|\breve{\Sigma}_N(i,j, i',j')-\hat{d}_{i,N}\hat{c}_{j,N} \hat{d}_{i',N}\hat{c}_{j',N}  {\Sigma}_N^*(i,j, i',j')\right| =o_P(1),
\end{equation}
where
$$
\breve{\Sigma}_N(i,j, i',j') = \sum_{k=-(N-1)}^{N-1} K(k/B_N) \hat{{ \phi}}_{k,N}(i,j, i',j')
$$
and
$$
{\Sigma}_N^*(i,j, i',j') = \sum_{k=-(N-1)}^{N-1} K(k/B_N) {{ \phi}}^*_{k,N}(i,j, i',j')
$$
with
$$
\hat{{\phi}}_{k,N}(i,j, i',j')= \frac{1}{N} \sum_{\ell=\max(1,1-k)}^{\min(N,N-k)} \hat{{\gamma}}_{\ell}(i,j)\hat{{\gamma}}_{\ell+k}(i',j'),
$$
$$
{{\phi}}^*_{k,N}(i,j, i',j')= \frac{1}{N} \sum_{\ell=\max(1,1-k)}^{\min(N,N-k)} {{\gamma}}^*_{\ell}(i,j){{\gamma}}_{\ell+k}^*(i',j'),
$$
and
$$
{\gamma}^*_{\ell}(i,j)=\la X_{\ell}, v_j\ra \la \hat{\epsilon}_{\ell},w_i\ra.
$$
Since
$$
\la X_\ell, v_j\ra \la \hat{\epsilon}_\ell, w_i\ra =\gamma_\ell(i,j)+\la X_\ell, v_j\ra \la \nu_\ell, w_i\ra,
$$
(\ref{e:eig-w}),  (\ref{e:eig-v}) and Lemma \ref{l:betahat} imply that

\begin{equation}\label{eq:what}
\left|\tilde{\Sigma}_N-{\Sigma}^*_N\right|= o_P(1).
\end{equation}
\end{proof}

We have seen in Theorem \ref{Th:bart-1} that $\left|\tilde{\boldsymbol \Sigma}_N - {\boldsymbol \Sigma}\right| = \op{1}$.  In \eqref{eq:who} and \eqref{eq:what} we have seen that $\left|\breve{\boldsymbol \Sigma}_N - {\boldsymbol \zeta}_N {\boldsymbol \Sigma}_N^*{\boldsymbol \zeta}_N\right| = \op{1}$ and $\left|\tilde{\Sigma}_N-{\Sigma}^*_N\right|= o_P(1)$.  Therefore, $\left|\hat{{\boldsymbol \Sigma}}_N- {\boldsymbol \Sigma}\right| = \op{1}$, completing the proof.


\end{document}